\documentclass[10pt, twoside]{amsart}

\usepackage[utf8]{inputenc}
\usepackage{hyperref}

\usepackage{graphicx}              
\usepackage{amsmath}               
\usepackage{amsfonts}              
\usepackage{amsthm}                
\usepackage{amssymb}
\usepackage{mathtools}             
\usepackage{tikz-cd}
\usepackage[all]{xy}
\usepackage{enumitem}

\usepackage[all]{xy}
\usepackage{amscd}
\usepackage{cleveref}
\crefformat{equation}{(#2#1#3)}
\crefname{defi}{definition}{definitions}
\Crefname{defi}{Definition}{Definitions}
\Crefname{theorem}{Theorem}{Theorems}
\allowdisplaybreaks
\usepackage{color}
\usepackage{fp}
\newcommand{\N}{\mathcal{N}}

\newcommand{\Qq}{\mathcal{Q}}

\newcommand{\ob}{\text{Ob}}

\newcommand{\M}{\mathcal{M}}
\newcommand{\Nn}{\mathcal{N}}
\newcommand{\Cat}{\mathbf{Cat}}
\newcommand{\Set}{\mathbf{Set}}
\newcommand{\p}{\mathcal{P}}
\newcommand{\id}{\text{id}}
\newcommand{\la}{\langle}
\newcommand{\ra}{\rangle}

\DeclarePairedDelimiterX{\normb}[1]{\lVert}{\rVert}{#1}

\numberwithin{equation}{section}
\newtheorem{theorem}{Theorem}[section]
\newtheorem{lem}[theorem]{Lemma}

\newtheorem{cor}[theorem]{Corollary}

\let\origproofname\proofname
\renewcommand{\proofname}{\upshape\textbf{\origproofname}}

\theoremstyle{definition}
\newtheorem{ex}[theorem]{Example}
\newtheorem{defi}[theorem]{Definition}
\newtheorem{rmk}[theorem]{Remark}

\usepackage{indentfirst}

\begin{document}

\title{A coherence theorem for pseudo symmetric multifunctors}

\author{Diego Manco}

\date{\today}

\address{Department of Mathematics, University  of Oregon,
       Eugene OR 97403-1222, USA.}

\email[]{dmancobe@uoregon.edu}

\thanks{\textsc{Department of Mathematics, University of Oregon, Eugene OR 97403-1222, USA}
}

\maketitle
\begin{abstract}
In \cite{Y23} Yau defines the notion of pseudo symmetric $\Cat$-enriched multifunctor between $\mathbf{Cat}$-enriched multicategories and proves that Mandell's inverse $K$-theory multifunctor \cite{Mandell10} is pseudo symmetric. We prove a coherence theorem for pseudo symmetric $\mathbf{Cat}$-multifunctors.  As an application we prove that pseudo symmetric $\Cat$-multifunctors, and in particular Mandell's inverse $K$-theory, preserve $E_n$-algebras ($n=1,2,...,\infty$), at the cost of changing the parameterizing $E_n$-operad. 
\end{abstract}
\section{Introduction}
Permutative categories are symmetric monoidal categories that are strictly associative and unital. Let $\mathbf{Perm}$ be the category of permutative categories.  By a construction of May \cite{May74}, we can define algebraic $K$-theory as a functor from $\mathbf{Perm}$ to spectra. Elmendorf and Mandell \cite{EM06} introduced multicategories in homotopy theory to study the multiplicative properties of this functor. They gave $\mathbf{Perm}$ the structure of a multicategory and showed that the $K$-theory construction can be extended to a symmetric multifunctor landing in spectra. This implies that $K$-theory preserves certain multiplicative structures---for example, the $K$-theory of a bipermutative category is an $E_\infty$ ring spectrum.
\vspace*{-0.2cm}\\

Following work of Thomason \cite{T95}, Mandell \cite{Mandell10} introduced inverse $K$-theory $\p$, a functor from $\Gamma$-categories (modelling connective spectra) to $\mathbf{Perm}$ that provides a homotopy inverse to $K$-theory.  Elmendorf \cite{E21} and Johnson-Yau \cite{JY22} extended $\p$ to a $\mathbf{Cat}$-enriched multifunctor, but one that is not symmetric: it is not compatible with the permutation of elements in the domains of multicategory mapping spaces.  To account for this Yau \cite{Y23} introduced pseudo-symmetric multifunctors, where there is a compatibility only up to coherent natural isomorphisms, and he proved that Mandell's inverse $K$-theory multifunctor $\p$ is pseudo symmetric in his sense.
\vspace*{-0.2cm}\\

In this article we establish a 2-adjunction that lets us rigidify pseudo symmetric multifunctors and write them as symmetric multifunctors at the cost of fattening up their domain in a specific way. As an application we get a new result in multiplicative $K$-theory: pseudo symmetric multifunctors, and in particular Mandell's inverse $K$-theory,  preserve $E_n$-algebras for $n=1,2,\dots ,\infty$ at the cost of changing the parameterizing $E_n$ operad. For example, they send commutative monoids to $E_\infty$ algebras.\\

Let us go back and provide more details of the above panorama. Segal's infinite loop space machine \cite{Segal74} allows the construction of spectra from symmetric monoidal categories. May's construction \cite{May74} provides an alternative way of building spectra from permutative categories. Both $K$-theory constructions turn out to be equivalent \cite{MT78}, with $\mathbf{Perm}$ being equivalent to the category of symmetric monoidal categories by a theorem of Isbell \cite{I69}. The question about what kind of structure to impose on a permutative category so that its $K$-theory is an $E_\infty$ ring spectrum was answered independently by Elmendorf and Mandell \cite{EM06} and May \cite{May09}, the former using the theory of multicategories. To study multiplicative $K$-theory, one would like the domain of the $K$-theory construction $\mathbf{Perm}$ to have a symmetric monoidal structure and $K$-theory to be a monoidal functor. That way, $K$-theory would preserve multiplicative structures in $\mathbf{Perm}$. However, $\mathbf{Perm}$ lacks a natural symmetric monoidal structure, although it admits one in a 2-categorical sense \cite{GJO22}.
\vspace*{-0.2cm}\\

Multicategories, also known as colored operads, generalize symmetric monoidal categories by supplying a setup for working with multi-input maps, thus providing an alternative way of encoding multiplicative structures even in the absence of symmetric monoidal structures. In a sense, multicategories allow us to talk about multilinear maps without making any reference to tensor products. Multiplicative structures can then be encoded in a multicategory via the actions of operads and similar gadgets. Elmendorf-Mandell \cite{EM06} gave $\mathbf{Perm}$ the structure of a multicategory and extended algebraic $K$-theory to a symmetric multifunctor landing in symmetric spectra. This implies that $K$-theory preserves multiplicative structures. This is how they proved that the $K$-theory of a bipermutative category is an $E_\infty$ ring spectrum. Multiplicative $K$-theory has also been defined as a symmetric multifunctor from the multicategory of Waldhausen categories $\mathbf{Wald}$ to spectra \cite{BM11}, with $\mathbf{Wald}$ providing another example of a multicategory that doesn't arise from a symmetric monoidal structure \cite{Z18}.\\ 

Spectra arising from the Segal-May construction are all connective, and by a theorem of Thomason \cite{T95} the $K$-theory construction is surjective on homotopy types. Mandell's inverse $K$-theory functor $\p\colon \Gamma\mbox{-}\Cat\to \mathbf{Perm}$ witnesses this by providing a homotopy stable inverse to $K$-theory. Here $\Gamma$-categories model connective spectra by \cite{T80,C99,BF78}. Elmendorf \cite{E21} and Johnson-Yau \cite{JY22} prove independently that Mandell's inverse $K$-theory functor can be extended to a $\Cat$-enriched multifunctor $\p\colon \Gamma\mbox{-}\Cat\to\mathbf{Perm}$ between $\Cat$-enriched multicategories. However, $\p$ turns out to not be symmetric \cite{JY22}, i.e., it doesn't preserve the action of the symmetric group on the hom objects of the multicategories by permutation of inputs.  So their results can only be used to prove that $\p$ preserves multiplicative structures that don't involve symmetry, like associative monoids \cite{JY22}. This obstruction led Yau \cite{Y23} to define pseudo symmetric multifunctors. These are non-symmetric $\Cat$-enriched multifunctors that preserve the action of the symmetric group of multicategory mapping spaces only up to coherent natural isomorphisms. One of the main results of \cite{Y23} is that $\p$ is in fact pseudo symmetric.\\

Our main result can be interpreted as a coherence result for pseudo symmetric multifunctors. If $F\colon \M\to\Nn$ is a pseudo symmetric multifunctor between $\Cat$-enriched multicategories, we prove that the natural isomorphisms attesting the pseudo symmetry of $F$ assemble together to give a symmetric $\Cat$-enriched multifunctor $\phi(F)\colon \M\times E\Sigma_*\to\Nn$ satisfying a universal property, where $E\Sigma_*$ is the categorical Barratt-Eccles operad defined in \Cref{Barrat}. We can also think about our result as a rigidification result. We can rigidify $F$ and turn it into a symmetric $\Cat$-enriched multifunctor $\phi(F),$ at the cost of changing its domain.
 
\begin{theorem}\label{mainresult}\normalfont{(\Cref{mainprop})} Let $\mathcal{M}$ be a $\textbf{Cat}$-enriched multicategory. There is a pseudo symmetric multifunctor $\eta_\M\colon \mathcal{M}\rightarrow \mathcal{M}\times E\Sigma_*$ such that for every $\textbf{Cat}$-enriched multicategory $\mathcal{N}$ and each pseudo symmetric multifunctor $F\colon \mathcal{M}\rightarrow \mathcal{N},$ there exists a unique symmetric $\Cat$-enriched multifunctor $\phi(F)\colon \mathcal{M}\times E\Sigma_*\rightarrow \mathcal{N}$ such that the following diagram commutes: 
\begin{center}    
\begin{tikzcd}
        &\M\times E\Sigma_*\ar[rd,"\phi(F)"]&\\
        \M\ar[rr,"F"swap]\ar[ru,"\eta_\M"]&&\Nn.\\
    \end{tikzcd}
\end{center} That is, $F=\phi(F)\circ \eta_\M$ as pseudo symmetric multifunctors.
\end{theorem}
Thus, if $\mathcal{O}$ is an operad in $\Cat,$ pseudo symmetric algebras in a $\Cat$-enriched multicategory $\M$ over $\mathcal{O},$ i.e., pseudo symmetric multifunctors $\mathcal{O}\to \M,$ are symmetric algebras in $\M$ over $\mathcal{O}\times E\Sigma_*.$ The following result, which appears as \Cref{finalcor}, holds  since multiplying by $E\Sigma_*$ sends the commutative operad $\{*\}$ to the $E_\infty$ operad $E\Sigma_*$ and $E_n$ operads in $\Cat,$ like the ones defined in \cite{Berger96} and \cite{Fiedorowicz03}, to $E_n$ operads.
\begin{cor}\label{maincor}\normalfont{(\Cref{finalcor})} Let $F\colon \M\to\N$ be $\Cat$-enriched pseudo symmetric multifunctor. Then,
\begin{enumerate} 
\item  $F$ sends commutative monoids to $E_\infty$ algebras.
\item $F$ sends $E_n$ algebras to $E_n$ algebras for  $n=1,2,\dots,\infty.$
\end{enumerate}
\end{cor}
This corollary extends our understanding of the behavior of inverse $K$-theory since it implies that the inverse $K$-theory multifunctor $\p$, which is pseudo symmetric by work of Yau \cite{Y23},  sends commutative monoids to $E_\infty$-algebras and preserves $E_n$ algebras ($n=1,2,\dots$). Since $\p$ provides a stable inverse to $K$-theory, and $K$-theory is a symmetric multifunctor, this implies that every $E_n$-algebra in $\Gamma$-categories is stably equivalent to the $K$-theory of an $E_n$ algebra in permutative categories. This shows how \Cref{mainresult} can be used to grasp the behavior of pseudo symmetric multifunctors on structures parameterized by symmetric operads in general. \\

In \cite{Y23} Yau defines the $2$-category $\mathbf{Cat}$-$\mathbf{Multicat}$ having $\Cat$-enriched multicategories as $0$-cells, symmetric multifunctors as 1-cells  and multinatural transformations as 2-cells. He also defines the $2$-category $\mathbf{Cat}$-$\mathbf{Multicat^{ps}}$ with 0-cells  $\Cat$-enriched multicategories, 1-cells pseudo symmetric multifunctors, and 2-cells pseudo symmetric $\Cat$-multinatural transformations. Every symmetric $\Cat$-enriched multifunctor (respectively multinatural transformation) is canonically a pseudo symmetric multifunctor (respectively multinatural transformation), so there is a 2-functorial inclusion $j\colon \mathbf{Cat}$-$\mathbf{Multicat}\to\mathbf{Cat}$-$\mathbf{Multicat^{ps}}.$ Taking into account these 2-categorical structures we can improve our previous result by providing a left adjoint $\psi$ to $j,$ which, at the 0-cell level, sends a multicategory $\M$ to $\psi(\M)=\M\times E\Sigma_*.$ 
\begin{theorem}\label{adjunction}\normalfont{(\Cref{phi} and \Cref{finaladjunction})}
    The inclusion $j\colon\mathbf{Cat}$-$\mathbf{Multicat}\to\mathbf{Cat}$-$\mathbf{Multicat^{ps}}$ admits a left 2-adjoint $\psi\colon\mathbf{Cat}$-$\mathbf{Multicat^{ps}}\to\mathbf{Cat}$-$\mathbf{Multicat}$ with $\psi(\M)=\M\times E\Sigma_*$ for $\M$ a $\Cat$-multicategory. In particular, for $\Cat$-multicategories $\M$ and $\N$ we have an isomorphism of categories
    $$\mathbf{Cat}\mbox{-}\mathbf{Multicat^{ps}}(\mathcal{M},\mathcal{N})\cong \mathbf{Cat}\mbox{-}\mathbf{Multicat}(\mathcal{M}\times E\Sigma_*,\mathcal{N}).$$
\end{theorem}
An important consequence of this theorem is that we can give a very simple and compact description of the 2-category $\Cat\mbox{-}\mathbf{Multicat^{ps}}$ solely in terms of symmetric $\Cat$-multifunctors and $\Cat$-mutinatural transformations, which we do in \Cref{2D}.\\

Bohmann and Osorno \cite{BO16} make use of a spectrally enriched version of the Elmendorf-Mandell construction together with the description of equivariant spectra in terms of presheaves of spectra due to Guillou and May \cite{GM11} to define an equivariant infinite loop space machine. Since preservation of multiplicative structures is one of the main ingredients in the construction of this equivariant machine \cite{BO16}, our results can also be regarded as a step towards proving the conjecture that every connective equivariant spectrum, i.e., those whose fixed point spectra are connective, arises from Bohmann and Osorno's construction. The infinite equivariant loop space machine $\mathbb{K}_G$ from \cite{GMMO21} is also suspected to be pseudo symmetric, so our result might help understand the preservation of multiplicative structures in the equivariant context as well.\\

\textbf{Outline.} In  \Cref{Section2}  we recall the definition of the 2-categories $\Cat\mbox{-}\textbf{Multicat}$ and $\Cat\mbox{-}\mathbf{Multicat^{ps}}.$ In \Cref{Section3} we prove \Cref{mainresult,adjunction}. We also extract a new and compact description of the 2-category $\Cat\mbox{-}\mathbf{Multicat^{ps}}$. In \Cref{Section4} we obtain the desired consequences for Mandell's inverse $K$-theory functor  $\p$  included in \Cref{maincor}.\\

\textbf{Acknowledgements.} The author would like to thank Angélica Osorno, without whom this article wouldn't exist, for her mentorship and guidance during all stages of this project. The author would also like to thank Dan Dugger for providing helpful advice and feedback as well as Donald Yau for his helpful comments, questions, email exchanges, and encouragement. The author was supported by the Fulbright-COLCIENCIAS scholarship awarded by the Fulbright Colombia Commission and COLCIENCIAS, now a part of the Ministry of Science, Technology, and Innovation of the Colombian government. The author was also supported by the University of Oregon through the Anderson Mathematics PhD Student Research Award.

\section{Symmetric and pseudo symmetric Multifunctors}\label{Section2}
We begin by reviewing the definition of multicategory enriched in a symmetric monoidal category. In the following definition $(C,1,\oplus,\lambda,\rho,\xi)$ is a symmetric monoidal category with $\oplus\colon C\times C\to C$ the monoidal product, 1 the monoidal unit, $\lambda$ the left unit isomorphism, $\rho$ the right unit isomorphism and $\xi$ the symmetry. In this paper we will consider only categories enriched over $\Cat$ with the monoidal structure given by products, but we use a general monoidal category in the definition to make explicit the fact that this definition doesn't make use of the 2-categorical structure of $\Cat$. 

\begin{rmk}\label{gam} We will also use the following notation: if $\sigma \in \Sigma_n$ and $\tau_i\in \Sigma_{k_{i}}$ for $1\leq i \leq n,$ $\sigma\langle \tau_1,\dots, \tau_n\rangle\in \Sigma_{k_1+\cdots+k_n}$ is the permutation that permutes $n$ blocks of lengths $k_1,\dots, k_n$ according to $\sigma$ and each block of length $k_i$ according to $\tau_i.$  
\end{rmk}
\begin{defi}\label{multicat}
    If $C$ is a symmetric monoidal category, a $C$-\textit{multicategory} $(\M,\gamma,1)$ consists of the following data.
    \begin{itemize}
        \item A class of objects $\ob(\M).$
        \item For every $n\geq 0,$ $\la a\ra=\la a_i\ra_{i=1}^n\in\ob(\M)^n$  and $b\in \ob(\M),$ an object in $C$ denoted by
        $$\M(\la a\ra;b)=\M(a_1,\dots,a_n;b).$$
        We will write $\la a\ra$ instead of $\la a_i\ra_{i=1}^n$ when $n$ is clear from the context or irrelevant. [In the case $C=\Cat,$ an object $f$ of $\M(\la a\ra;b)$ will be called an $n$-ary 1-cell with input $\la a\rangle$ and output $b$ and will be denoted as $f\colon \la a \ra \to b.$ Similarly, we will call $\alpha\colon f\to g$ in $\M(\la a\ra;b)(f,g)$ an $n$-ary 2-cell.]
        \item For each $n\geq 0,$  $\la a\ra\in\ob(\M)^n,$ $b\in\ob(\M),$ and $\sigma\in \Sigma_n,$ a $C$-isomorphism
        \begin{center}
            \begin{tikzcd}
                \M(\la a\ra;b)\ar[r,"\sigma","\cong"swap]&\M(\la a\ra \sigma;b)
            \end{tikzcd}
        \end{center}
        called the right $\sigma$ action or the symmetric group action. Here
        $$\la a\ra\sigma=\la a_1,\dots,a_n\ra \sigma=\la a_{\sigma(1)},\dots,a_{\sigma(n)}\ra.$$ [In the case $C=\Cat$ we write $f\sigma$ for the image of an $n$-ary 1-cell $f\colon \la a\ra \to b$ in $\M$ and similarly for 2-cells.]
        \item For each object $a\in\ob(\M),$ a morphism
        \begin{center}
            \begin{tikzcd}
                1\ar[r,"1_a"]&\M(a;a)
            \end{tikzcd}
        \end{center}
        called the $a$-unit. In the case $C=\Cat$ we notice that  if $a\in\ob(\M),$ $1_a\colon a\to a$ is a 1-ary 1-cell while if $f\colon \la a\ra\to b$ is an $n$-ary 1-cell, then $1_f\colon f\to f$ is an $n$-ary 2-cell in $\M(\la a\ra;b)(f,f)$ so our notation is unambiguous.
        \item For every $c\in\ob(\M),$ $n\geq 0,$ $\la b \ra =\la b_j\ra_{j=1}^n\in \ob(M)^n,$ $k_j\geq 0$ for $1\leq j\leq n,$   and $\la a_j\ra=\la a_{j,i}\ra_{i=1}^{k_j}\in \ob(\M)^{k_j}$ for $1\leq j \leq n,$   a morphism in $C,$
        \begin{center}
            \begin{tikzcd}
                \M(\la b\ra;c)\otimes \bigotimes\limits_{j=1}^n\M(\la a_j\ra;b_j)\ar[r,"\gamma"]&\M(\la a\ra;c),
            \end{tikzcd}
        \end{center}
        where we adopt the convention that $\la a\ra \in \ob(\M)^k,$ where $k=\sum_{i=1}^n k_j,$ denotes the concatenation of the varying $a_j$'s for $j=1,\dots,n$. We write this as
        $$\la a\ra =\la a_1,\dots,a_n\ra=\la\la a_j\ra\ra_{j=1}^n= \la a_{1,1},\dots, a_{1,k_1},a_{2,1,}\dots,a_{n-1,k_{n-1}}a_{n,1},\dots,a_{n,k_n}\ra.$$
    \end{itemize}
    The previous data are required to satisfy the following axioms.
    \begin{itemize}
        \item \textbf{Symmetric group action}: For every $n\geq 0,$ $\la a\ra\in \ob(\M),$  $b\in\ob(\M),$ and $\sigma,\tau$ in  $\Sigma_n$ the following diagram commutes in $C:$
        \begin{center}
            \begin{tikzcd}
                \M(\la a\ra;b)\ar[r,"\sigma"]\ar[rd,"\sigma\tau"swap]&\M(\la a\ra\sigma;b)\ar[d,"\tau"]\\
                &\M(\la a\ra\sigma\tau;b).
            \end{tikzcd}
        \end{center}
        We also require the identity permutation $\id_n\in\Sigma_n$ to act as the identity morphism on $\M(\la a\ra;b).$
        
        \item \textbf{Associativity:} For every $d\in\ob(\M),$ $  n\geq 1,$ $\la c\ra=\la c_j\ra_{j=1}^n\in\ob(\M)^n,$ $k_j\geq 0$ for $1\leq j\leq n$ with $k_j\geq 1$ for at least one $j,$ $\la b_j\ra=\la b_{j,i}\ra_{i=1}^{k_j}\in\ob(\M)^{k_j}$ for $1\leq j\leq n,$ $l_{i,j}\geq 0$ for $1\leq j\leq n $ and $1\leq i\leq k_j,$   and $\la a_{j,i}\ra=\langle a_{j,i,p}\rangle_{p=1}^{l_{i,j}}\in\ob(\M)^{l_{i,j}}$ for $1\leq j \leq n$ and $1\leq i\leq k_j,$ the following \textit{associativity diagram} commutes in $C$:
       
        \begin{equation}\label{gamma}
            \begin{tikzcd}[font=\small]
                \M(\la c\ra;d)\otimes\left(\bigotimes\limits_{j=1}^n\M(\langle b_j\rangle;c_j)\right)\otimes \bigotimes\limits_{j=1}^n\left(\bigotimes\limits_{i=1}^{k_j}\M(\langle a_{j,i}\rangle;b_{j,i})\right)\ar[rd,"\gamma\otimes 1"]\ar[dd,"\cong"]&[-95pt]\\
                &\M(\langle b\rangle;c)\otimes \bigotimes\limits_{j=1}^n\left(\bigotimes\limits_{i=1}^{k_j}\M(\langle a_{j,i}\rangle);b_{j,i}\right)\ar[dd,"\gamma"]\\
                \M(\la c\ra;d)\otimes \bigotimes\limits_{j=1}^n\left(\M(\langle b_j\rangle;c_j)\otimes \bigotimes\limits_{i=1}^{k_j}\M(\langle a_{j,i}\rangle;b_{j,i})\right)\ar[d,"1 \otimes\bigotimes_{j=1}^n\gamma"swap]&\\
                \M(\langle c\rangle;d)\otimes\bigotimes\limits_{j=1}^n\M(\la a_j\ra;c_j)\ar[r,"\gamma"swap]&\M(\langle a\rangle;b).
            \end{tikzcd}
        \end{equation}
        
        \item \textbf{Unity:} Suppose $b\in \ob(\M)$ and $\la a\ra=\la a_j\ra_{j=1}^n\in\ob(\M),$ then the following \textit{right unity diagram} commutes in $C:$
        \begin{center}
            \begin{tikzcd}
                \M(\langle a\rangle;b)\otimes\bigotimes\limits_{j=1}^n1\ar[d,"\id \otimes \bigotimes\limits_{j=1}^n1_{a_j}"swap]\ar[rd,"\cong"]& \\
                \M(\la a\ra;b)\otimes\bigotimes\limits_{j=1}^n\M(a_j;a_j)\ar[r,"\gamma"swap]&\M(\la a\ra;b).
            \end{tikzcd}
        \end{center}
        With $b,\la a\ra$ as before, we also demand that the following \textit{left unity diagram} commutes in $C.$
        \begin{center}
            \begin{tikzcd}
                1\otimes M(\la a\ra;b)\ar[rd,"\lambda"]\ar[d,"1_b\otimes \id"swap]& \\
                \M(b;b)\otimes \M(\langle a\rangle;b)\ar[r,"\gamma"swap]&\M(\la a\ra;b).
            \end{tikzcd}
        \end{center}
    
        \item \textbf{Top equivariance:} For every $c\in\ob(\M),$ $n\geq 1,$ $\la b \ra =\la b_j\ra_{j=1}^n\in\ob(\M)^n,$ $k_j\geq 0$ for $1\leq j\leq n,$ $\la a_j\ra=\la a_{j,i}\ra_{i=1}^{k_j}\in \ob(\M)^{k_j}$ for $1\leq j \leq n,$ and $\sigma\in \Sigma_n,$ the following diagram commutes:
        \begin{equation}\label{top}
            \begin{tikzcd}
                \M(\la b\ra;c)\otimes\bigotimes\limits_{j=1}^n\M(\la a_j\ra;b_j)\ar[r,"\sigma\otimes \sigma^{-1}"]\ar[d,"\gamma"swap]&[20pt] \M(\la b\ra\sigma;c)\otimes\bigotimes\limits_{j=1}^n\M(\la a_{\sigma(j)}\ra;b_{\sigma(j)})\ar[d,"\gamma"]\\
                \M(\la a_1\ra,\dots,\la a_n\ra;c)\ar[r,"\sigma\bigl\langle\id_{k_{\sigma(1)}}{,}\dots{,}\id_{k_{\sigma(n)}}\bigr\rangle"swap]&\M(\la a_{\sigma(1)}\ra,\dots,\la a_{\sigma(n)}\ra;c).
            \end{tikzcd}
        \end{equation}
        Here $\sigma^{-1}$ is the unique isomorphism in $C,$ given by the coherence theorem for symmetric monoidal categories, that permutes the factors $\M(\la a_j\ra,b_j)$ according to $\sigma^{-1}.$
        \item \textbf{Bottom equivariance:} For $\la a_j\ra,\la b\ra$ and $c$ as in Top equivariance \Cref{top}, the following diagram commutes:
        \begin{equation}\label{bottom}
            \begin{tikzcd}
                \M(\la b\ra;c)\otimes\bigotimes\limits_{j=1}^n\M(\la a_j\ra;b_j)\ar[r,"\id\otimes \bigotimes\limits_{j=1}^n\tau_j"]\ar[d,"\gamma"swap]&[20pt] \M(\la b\ra,c)\otimes\bigotimes\limits_{j=1}^n\M(\la a_j\ra\tau_j;b_j)\ar[d,"\gamma"]\\
                \M(\la a_1\ra,\dots,\la a_n\ra;c)\ar[r,"\id_n\bigl\langle\tau_1{,}\dots{,}\tau_n\bigr\rangle"swap]&\M(\la a_1\ra\tau_1,\dots,\la a_n\ra\tau_n;c).
            \end{tikzcd}
        \end{equation}
    \end{itemize}
    This concludes the definition of a $C$-multicategory.
\end{defi}
\begin{rmk} A $C$-operad is a $C$-multicategory with one object. If $\mathcal{O}$ is a $C$-operad, its $n$-ary operations will be denoted by $\mathcal{O}_n\in \ob(C).$ A non symmetric $C$-multicategory ($C$-operad) is defined in the same way as a  $C$-multicategory ($C$-operad) excluding the data of the symmetric group action as well as the symmetric group, top and bottom equivariance coherence axioms. We will only be concerned with symmetric multicategories and operads. $C$-multicategories are often referred to as colored operads, with the objects of the $C$-multicategory being referred to as colors and $C$-operads having just one color.
\end{rmk}
\begin{ex}
As examples of $\Set$-operads, where $\Set$ has the monoidal structure induced by products in $\Set,$ we have the commutative operad $\text{Comm}=\{*\}$ with $\text{Comm}_n=\{*\}$.  Another example is the associative operad $\text{Ass}=\Sigma_*$ with $\text{Ass}_n=\Sigma_n,$ with the right action of the symmetric product given by right multiplication and $\gamma$ defined in the following way. If $n\geq 1$ and $k_1,\dots,k_n$ natural numbers with $k=\Sigma_{i=1}^nk_i,$ we define $\gamma\colon \Sigma_n\times\left(\prod_{i=1}^n\Sigma_{k_i}\right)\to \Sigma_k$ given for $\sigma\in \Sigma_n$ and $\la\tau_1,\dots,\tau_n\ra\in \prod_{i=1}^n\Sigma_{k_i}$ by
$$\gamma(\sigma,\la \rho_i\ra_{i=1}^n)=\sigma\la \rho_i\ra_{i=1}^n=\sigma\langle \rho_1,\dots,\rho_n\rangle,$$
as in \Cref{gam}. When $n$ is clear from the context we will write $\sigma\la\rho_i\ra=\sigma\la \rho_i\ra_{i=1}^n.$
\end{ex}
\begin{ex}\label{Barrat}
We will consider $\Cat$-multicategories where the monoidal structure in $\Cat$ is given by products. One source of examples is the forgetful functor $\ob\colon \Cat\to\Set$ which forgets the morphism structure and remembers only the object set.  Its right adjoint $E\colon \Set\to \Cat$ is the functor that takes a set $A$ to $EA,$ the category with objects $\ob(EA)=A,$ and with a unique isomorphism between each pair of objects. $E$ sends a morphism $f\colon A\to B$ of sets to the functor $Ef\colon EA\to EB,$ the only functor such that $f=\ob(Ef)$. $E$ preserves products, and thus, if $\mathcal{O}$ is a  $\Set$-operad, $E\mathcal{O}$ is a $\Cat$-operad. Similarly, if $\M$ is a $\mathbf{Set}$-multicategory, $E\M$ is a $\Cat$-multicategory with the same collection of objects as $\M$.\\

We will call $E \text{Comm}=\{*\}$ the commutative $\Cat$-operad. The Barratt-Eccles operad is the \textbf{Cat}-operad $E\Sigma_*=EAss$. $\mathbf{Cat}$-algebras over this $E\Sigma_*$ are precisely permutative categories \cite{May74}.
\end{ex}
\begin{ex}
Another source of examples for multicategories are symmetric monoidal categories, and thus also permutative categories. Each symmetric monoidal category $C$ has an associated $\Set$-multicategory $End(C),$ whose objects agree with the objects of $C$ and such that for $\la a\ra\in \ob(C)^n$ and $b\in \ob(C),$
$$End(C)(\la a\ra;b)=C(a_1\otimes\cdots\otimes a_n,b).$$
Here we take $a_1\otimes \cdots \otimes a_n$ with the leftmost parenthesization. Any fixed parenthesization would work. An empty string of objects is interpreted as the monoidal unit $1\in \ob(C).$    
\end{ex}

Next, we define 1-cells between $C$-multicategories that preserve the action of the symmetric group. These are called symmetric $C$-multifunctors. 
\begin{defi}\label{multifunc}
    A (symmetric) $C$-multifunctor $F\colon \M\to\Nn$ between $C$-multicategories $\M$ and $\Nn$ consists of the following data.
    \begin{itemize}
        \item An object assignment $F\colon \ob(\M)\to\ob(\Nn)$.
        \item For each $n\geq 0,$ $\la a\ra\in\ob(\M)^n$ and $b\in\ob(\M)$ a $C$ morphism
        \begin{center}
            \begin{tikzcd}
                \M(\la a\ra;b)\ar[r,"F"]&\Nn(\la Fa\ra;Fb).
            \end{tikzcd}
        \end{center}
    \end{itemize}
    These data are required to preserve units, composition, and the action of the symmetric group.
    \begin{itemize}
        \item\label{1axmultif} \textbf{Units:} For each object $a\in\ob(\M),$ 
        $F(1_a)=1_{Fa},$ i.e., the following diagram commutes in $C:$
        \begin{center}
            \begin{tikzcd}
                &[10pt]\M(a,a)\ar[dr,"F"]&\\
                1\ar[ru,"1_a"]\ar[rr,"1_{Fa}"]&&\N(Fa,Fa).
            \end{tikzcd}
        \end{center}
        \item \textbf{Composition:} For every $c\in\ob(\M),$ $n\geq 0,$ $\la b \ra =\la b_j\ra_{j=1}^n\in \ob(M)^n,$ $k_j\geq 0$ for $1\leq j\leq n,$   and $\la a_j\ra=\la a_{j,i}\ra_{i=1}^{k_j}\in \ob(\M)^{k_j}$ for $1\leq j \leq n$ and $1\leq i \leq n,$ the following diagram commutes in $C:$
        \begin{center}
            \begin{equation}\label{2axmultif}
                \begin{tikzcd}
                    \M (\la b\ra;c)\otimes\bigotimes\limits_{j=1}^n\M(\langle a_j\rangle;b_j)\ar[r,"F\otimes\bigotimes\limits_{j=1}^n F"]\ar[d,"\gamma"{swap,name=h}]&[5pt]\Nn(\la Fb\ra;Fc)\otimes\bigotimes\limits_{j=1}^n\Nn(\langle Fa_j\rangle;Fb_j)\ar[d,"\gamma"name=g]\\
            \M(\la a\ra;c)\ar[r,"F"swap]&|[alias=N1]|\Nn(\la Fa\ra;Fc).
                \end{tikzcd}
            \end{equation}
        \end{center}
        \item\textbf{Symmetric Group Action:} For each $\la a \ra \in\ob(\M)^n$ and $b\in\ob(\M)$ the following diagram commutes in $C:$
        \begin{center}
            \begin{equation*}
                \begin{tikzcd}
                    \M(\la a\ra;b)\ar[r,"F"]\ar[d,"\sigma","\cong"swap]&\Nn(\la Fa\ra;Fb)\ar[d,"\sigma","\cong"swap]\\
                    \M(\la a\ra\sigma;b)\ar[r,"F"swap]&\N(\la Fa\ra\sigma;Fb).
                \end{tikzcd}
            \end{equation*}
        \end{center}
    \end{itemize}
    \end{defi}
    Next we define composition of $C$-multifunctors.
    \begin{defi}\label{1comp}
    \begin{itemize}
        \item  We define the \textit{horizontal composition} of $C$-multifunctors in the following way. Let $F\colon \M \to\Nn,$ and $G\colon \Nn\to \mathcal{Q}$ be $C$-multifunctors, we define the $C$-multifunctor $GF\colon \M\to \mathcal{Q}$ \cite{Y23} on objects as the composition
        \begin{center}
            \begin{tikzcd}
        \ob(\M)\ar[r,"F"]&\ob(\Nn)\ar[r,"G"]&\ob(\mathcal{Q}),
            \end{tikzcd}
        \end{center}
        and its component functors for $\la a\ra\in\ob(\M)^n$, $b\in\ob(\M)$ as the composite
        \begin{center}
            \begin{tikzcd}
                \M(\la a\ra;b)\ar[r,"F"]&\Nn(\la Fa\ra ;Fb)\ar[r,"G"]&\mathcal{Q}(\la GFa\ra;GFb).
            \end{tikzcd}
        \end{center}
        \item The identity $C$-multifunctor $1_\M\colon \M\to\M$ is defined as the identity assignment on objects with the identity functors as component functors.
    \end{itemize}
\end{defi}
Next we define 2-cells between $C$-multifunctors. These will be the 2-cells of a 2-category with 0-cells $C$-multicategories and $1$-cells $C$-multifunctors.
\begin{defi}\normalfont{(\cite{Y23}, Def. 3.2.5)}
    For (symmetric) $C$-multifunctors $F,G\colon \M\to \Nn,$ we define  a $C$-\textit{multinatural transformation} $\theta\colon F\Rightarrow G$ as the data of a component morphism $\theta_a\colon 1\to\Nn(Fa,Ga)$ in $C$ for each $a\in\ob(\M)$ subject to the commutativity of the following diagram in $C$ for each $\la a\ra\in\ob(\M)^n$ and $b\in \ob(\M),$
    \begin{center}
    \begin{tikzcd}
        &[-20pt]1\otimes \M(\la a\ra;b)\ar[r,"\theta_b\otimes F"]&\N(Fb;Gb)\otimes\N(\la Fa\ra;Fb)\ar[rd,"\gamma"]&[-70pt]\\
        \M(\la a\ra,b)\ar[ru,"\cong"]\ar[rd,"\cong"swap]&&&\N(\la Fa\ra;Gb).\\
        &\M(\la a\ra;b)\otimes \bigotimes_{j=1}^n 1\ar[r,"G\otimes \bigotimes \theta_{a_j}"]&\N(\la Ga\ra;Gb)\otimes\bigotimes_{j=1}^n\N(Fa_j;Ga_j)\ar[ru,"\gamma"swap]&
    \end{tikzcd}
    \end{center}
    \end{defi}
    We define the identity multinatural transformation $1_F\colon F\to F$ to have = components $(1_F)_a=1_{Fa}$ for $a$ an object of $\M.$
    \begin{rmk}\label{gammanotation}
        When $C=\Cat,$ and given $F,G\colon \M\to\Nn$ $\Cat$-multifunctors and the data of a 1-ary 1-cell $\theta_a\colon Fa\to Ga$ for each $a\in\ob(M),$ the commutativity of the diagram in the previous definition means that for every $n\geq 0,$ $\la a\ra\in\ob(\M)^n,$ $b\in\ob(\M)$ and each 1-cell $f\colon \la a\ra\to  b,$ \begin{equation}\label{trans1}
        \gamma(Gf;\la \theta_{a_j}\ra)=\gamma(\theta_{b};Ff)
    \end{equation}
    holds in $\N(\la Fa\ra;Gb)$ and that,  for every 2-cell $\alpha\colon f\to g$ in $\M(\la a\ra;b)(f,g),$
    \begin{equation}\label{trans2}
        \gamma(G\alpha;\la 1_{\theta_{a_j}}\ra)=\gamma(1_{\theta_{b}};F\alpha)
    \end{equation}
    in $\N(\la Fa\ra;Gb).$ We can express \Cref{trans1}  diagrammatically as the commutativity of the square
    \begin{center}
        \begin{tikzcd}
            \la Fa\ra\ar[r,"\langle\theta_{a_j}\rangle "]\ar[d,"Ff"']&\langle Ga\rangle\ar[d,"Gf"]\\
            Fb\ar[r,"\theta_b"]&Gb,
        \end{tikzcd}
    \end{center}
    where the composition of adjacent 1-cells is done through $\gamma$ and a square represents an equality between composite 1-cells. In the same fashion, and using \Cref{trans1}, we can express \Cref{trans2} as the equality of multicategorical pasting diagrams
    \begin{center}
           \begin{tikzcd}
\la Fa\ra \arrow[dd, bend left=50, "Fg"{name=U}]\arrow[dd, bend right=50, "Ff"'{name=D}]\ar[r,"\langle \theta_{a_j}\rangle"]&\la Ga\ra\ar[dd,"Gg"]&[-20pt]&[-20pt] \la Fa\ra\ar[r,"\la\theta_{a_j}\ra"]\ar[dd,"Ff"']&\la Ga\ra\arrow[dd, bend left=50, "Gg"{name=A}]\arrow[dd, bend right=50, "Gf"'{name=B}]\\[-12pt]
&&\mathbf{=}&&\\[-12pt]
Fb\ar[r,"\theta_b"']&Gb&&Fb\ar[r,"\theta_b"]&Gb.
\arrow[Rightarrow, from=D, to=U,shorten <=2mm,shorten >=1mm,"F\alpha"]
\arrow[Rightarrow, from=B, to=A,shorten <=2mm,shorten >=1mm,"G\alpha"]
\end{tikzcd}
    \end{center}
    Here the concatenation of adjacent 2-cells is done through $\gamma,$ and an arrow labeled with the 1-cell $h$ is interpreted as the 2-cell $1_h:h\to h.$ For example, the left hand side diagram represents $\gamma(1_{\theta_b},F\alpha)$ while the right hand side represents $\gamma(G\alpha,\la \theta_{\alpha_j}\ra).$ The empty squares represent equalities between composite 1-cells.
    \end{rmk}
 Next, we define horizontal and vertical compositions of $C$-multinatural transformations.
\begin{defi}\label{2comp}\normalfont{(\cite{Y23}, Def. 3.2.7)}
    \begin{center}
Suppose given $\theta\colon F\Rightarrow G,$ $\zeta\colon G\Rightarrow H$  $C$-multinatural transformations with $F,G,H\colon \M\to \Nn$ $C$-multifunctors. The \textit{vertical composition} $\zeta\theta\colon F\Rightarrow H$ is defined as having as component at each $a\in\ob(\M)$ $(\zeta\theta)_a,$ the composite
    \begin{equation*}
        \begin{tikzcd}
            1\ar[r,"\cong"]&1\otimes 1\ar[r,"\zeta_a\otimes\theta_a"]&\N(Ga;Ha)\otimes \N(Fa;Ga)\ar[r,"\gamma"]&\Nn(Fa;Ha).
        \end{tikzcd}
        \end{equation*}
    \end{center}
     Suppose that $\theta\colon F\Rightarrow G$ and  $\zeta\colon F'\Rightarrow G'$ are $C$-multinatural transformations with  $F,G\colon \M\to\Nn$ and $F',G'\colon \N\to \mathcal{Q}$ $C$-multifunctors.  The \textit{horizontal composition} $\zeta*\theta\colon F'F\Rightarrow G'G$ is defined as the $C$-multinatural transformation with component at each $a\in\ob(\M),$  given by the composite
    \begin{center}
    \begin{equation*}
        \begin{tikzcd}[font=\tiny]
        1\ar[d,"\cong"]\ar[rr,"(\zeta *\theta)_a"]&&[-15pt]\mathcal{Q}(F'Fa;G'Ga)\\
            1\otimes 1\ar[r,"\zeta_{Ga}\otimes \theta_a"]&\mathcal{Q}(F'Ga;G'Ga)\otimes \N(Fa;Ga)\ar[r,"1 \otimes F'"]&\mathcal{Q}(F'Ga;G'Ga)\otimes\mathcal{Q}(F'Fa;F'Ga).\ar[u,"\gamma"]
        \end{tikzcd}
        \end{equation*}
    \end{center}
\end{defi}
\begin{rmk}
    When $C=\Cat$ and given $\theta\colon F\Rightarrow G,$ $\zeta\colon G\Rightarrow H$  $\Cat$-multinatural transformations with $F,G,H\colon \M\to \Nn$ $C$-multifunctors and $a\in\ob(\M),$
    \begin{equation}\label{vert}
        (\zeta\theta)_a=\gamma(\zeta_a,\theta_a.)
    \end{equation}
    On the other hand, if $\theta\colon F\Rightarrow G$ and  $\zeta\colon F'\Rightarrow G'$ are $\Cat$-multinatural transformations with  $F,G\colon \M\to\Nn$ and $F',G'\colon \N\to \mathcal{Q}$ $\Cat$-multifunctors,
    \begin{equation}\label{hor}
        (\zeta*\theta)_a=\gamma(\zeta_{Ga};F'\theta_a).
    \end{equation}
\end{rmk}
Yau proves in \cite{Y23} that  \Cref{multicat,multifunc,1comp,2comp} assemble together to give the 2-category $C$-$\mathbf{Multicat}$, with  0-cells consisting of $C$-multicategories, 1-cells symmetric $C$-multifunctors, and 2-cells $C$-multinatural transformations. There is  a non symmetric variant where we drop the requirement that the $C$-multifunctors preserve the symmetric group action, as well as dropping the coherence axioms related to the symmetric group action, but we won't refer to this 2-category again.\\

For the rest of the article we fix our symmetric monoidal category $C$ to be $\Cat,$ with the symmetric monoidal structure induced by products. In this context we can define a pseudo symmetric variant of this 2-category, namely $\Cat$-$\mathbf{Multicat^{ps}}$ using the 2-categorical structure of $\Cat$. The objects of $\Cat$-$\mathbf{Multicat^{ps}}$ are still $\Cat$-multicategories, but the 1-cells are pseudo symmetric $\Cat$-multifunctors:  $\Cat$-multifunctors where we only require that they preserve the symmetric group action up to coherent isomorphisms.

\begin{defi}\normalfont{(\cite{Y23} Def. 4.1.1)} Suppose that $\mathcal{M},\mathcal{N}$ are $\Cat$-multicategories. A  \textit{pseudo symmetric $\Cat$-multifunctor} $F\colon \mathcal{M}\to \mathcal{N}$ consists of the following data:
\begin{itemize}
    \item A  function on object sets $F\colon \ob(\mathcal{M})\to\ob(\mathcal{N}).$
    \item For each $\la a\ra\in \ob(\mathcal{M})^n$ and $b\in\ob(\mathcal{M}),$ a component functor
    \begin{center}
\begin{tabular}{c}
    \xymatrix{\mathcal{M}(\langle a\rangle;b)\ar[r]^-{F}&\mathcal{N}(\langle Fa\rangle;Fb).}
\end{tabular}
\end{center}
\item For each $\sigma\in \Sigma_n,$ $\langle a\rangle\in\ob(\M)^n,$ $b\in \ob(\M),$  a natural isomorphism $F_{\sigma,\langle a\rangle,b}$
\begin{center}
\begin{tikzcd}
  \mathcal{M}(\langle a\rangle;b)\ar[r,"F"]\ar[d,"\sigma"swap]&|[alias=N]|\mathcal{N}(\langle Fa\rangle;Fb)\ar[d,"\sigma"]\\
  |[alias=M]|\mathcal{M}(\langle a\rangle\sigma;b)\ar[r,"F"swap]&\mathcal{N}(\langle Fa\rangle\sigma;Fb).
  \arrow[Rightarrow,from=M, to=N,shorten >=4mm,shorten <=4mm,"F_{\sigma,\langle a\rangle,b}"swap,"\cong"]
\end{tikzcd}
\end{center}
When $\langle a\rangle$ and $b$ are clear from the context we write simply $F_\sigma,$ and if $f\in\ob(\M(\la a\ra,b))$ we will denote by $F_{\sigma,\la a\ra,b;f}=F_{\sigma;f}\colon F(f\sigma)\to F(f)\sigma$ the 2-cell in $\Nn(\la Fa\ra\sigma;Fb)$ corresponding to the component of $F_\sigma$ at $f.$  Naturality for $F_\sigma$ means that given $\alpha\colon f\to g$ a 2-cell in $\M(\la a\ra;b)(f,g),$ the following diagram commutes in $\N(\la F a\ra\sigma;b):$  
\begin{equation}\label{4.1.9}
\begin{tikzcd}
F(f\sigma)\ar[r,"F_{\sigma ;f}"]\ar[d,"F(\alpha \sigma)"swap]&F(f)\sigma\ar[d,"(F\alpha)\sigma"]\\
F(g\sigma)\ar[r,"F_{\sigma;g}"swap]&F(g)\sigma.
    \end{tikzcd}
\end{equation}
\end{itemize}
These data are subject to the same axioms of unit and composition preservation \Cref{2axmultif}  as a symmetric $\mathbf{Cat}$-multifunctor, but we replace the symmetric group action preservation axiom by the following four axioms.\\

\begin{itemize}
    \item \textbf{Unit permutation:} Let $n\geq 0,$ $\la a\ra\in\ob(\M)^n$ and $b\in\ob(\M),$ then
    \begin{center}
        \begin{equation}\label{pseudounit}
        F_{\text{id}_n,\langle a\rangle,b}=1_F.
    \end{equation}
    \end{center}
    \item \textbf{Product permutation:} This axiom expresses the coherence of the natural isomorphisms $F_{\sigma},$ for varying $\sigma,$ with respect to the symmetric group action. Let $n\geq 0,$ $\la a\ra\in\ob(\M)^n,$ $b\in\ob(M)$ and $\sigma,\tau \in \Sigma_n.$ Then, the following equality of pasting diagrams holds.
    \begin{center}
    \begin{tikzcd}
        \mathcal{M}(\langle a\rangle;b)\ar[r,"F"]\ar[d,"\sigma"swap]&|[alias=N1]|\mathcal{N}(\langle Fa\rangle;Fb)\ar[d,"\sigma"]&[-30pt]&[-30pt]\mathcal{M}(\langle a\rangle;b)\ar[r,"F"]\ar[dd,"\sigma\tau"swap]&|[alias=N3]|\mathcal{N}(\langle Fa\rangle;Fb)\ar[dd,"\sigma\tau"]\\
        |[alias=M1]|\mathcal{M}(\langle a\rangle\sigma;b)\ar[r,"F"swap]\ar[d,"\tau"swap]&|[alias=N2]|\mathcal{N}(\langle Fa\rangle\sigma;Fb)\ar[d,"\tau"]&=&&\\
         |[alias=M2]|\mathcal{M}(\langle a\rangle\sigma\tau;b)\ar[r,"F"swap]&\mathcal{N}(\langle Fa\rangle\sigma\tau;Fb)&&|[alias=M3]|\mathcal{M}(\langle a\rangle\sigma\tau;b)\ar[r,"F"swap]&\mathcal{N}(\langle Fa\rangle\sigma\tau;Fb).
         \arrow[Rightarrow,from=M1,to=N1,"F_\sigma"swap,shorten >=4mm,shorten <=4mm]
         \arrow[Rightarrow,from=M2,to=N2,"F_\tau"swap,shorten >=4mm,shorten <=4mm]
    \arrow[Rightarrow,from=M3,to=N3,"F_{\sigma\tau}"swap,shorten >=10mm,shorten <=8mm]
    \end{tikzcd}
    \end{center}
    Thus, for every 1-cell $f\in\ob(\M(\la a\ra;b)),$ the following diagram of 2-cells commutes in $\Nn(\la Fa\ra;Fb)$:
    \begin{equation}\label{4.1.10}
       \begin{tikzcd}
       &F(f\sigma)\tau\ar[rd,"(F_{\sigma;f})\tau"]&\\
F(f\sigma\tau)\ar[ru,"F_{\tau;f\sigma}"]\ar[rr,"F_{\sigma\tau;f}"swap]&&F(f)\sigma\tau .
    \end{tikzcd} 
    \end{equation}
    \item \textbf{Top equivariance:} For every $c\in\ob(\M),$ $n\geq 0,$ $\la b \ra =\la b_j\ra_{j=1}^n\in \ob(M)^n,$ $k_j\geq 0$ for $1\leq j\leq n,$   and $\la a_j\ra=\la a_{j,i}\ra_{i=1}^{k_j}\in \ob(\M)^{k_j}$ for $1\leq j \leq n$ and $1\leq i \leq n,$ and $\sigma\in \Sigma_n,$ the following two pasting diagrams are equal. \\
    \begin{center}
        \begin{tikzcd}[font=\tiny]
            \M (\la b\ra;c)\times\prod_{j=1}^n\M(\langle a_j\rangle;b_j)\ar[rr,"F\times\prod_i F"]\ar[d,"\gamma"{swap,name=h}]&[-15pt]&[-15pt]\Nn(\la Fb\ra;Fc)\times\prod_{j=1}^n(\langle Fa_j\rangle;Fb_j)\ar[d,"\gamma"name=g]\\
            
            \M(\la\la a_j\ra\ra_{j=1}^n;c)\ar[d,"\sigma\la\id_{k_{\sigma(j)}}\ra"swap]\ar[rr,"F"]&&|[alias=N1]|\Nn(\la\la Fa_j\ra\ra_{j=1}^n;Fc)\ar[d,"\sigma\la\id_{k_{\sigma(j)}}\ra"]\\
            |[alias=M1]|\M(\la\la a_{\sigma(j)}\ra\ra_{j=1}^n;c)\ar[rr,"F"swap]&&\Nn(\la\la Fa_{\sigma(j)}\ra\ra_{j=1}^n;Fc)\\[-10pt]
            
            & \mathbf{\parallel} &\\
            \M (\la b\ra;c)\times\prod_{j=1}^n\M(\langle a_j\rangle;b_j)\ar[d,"\sigma\times 
            \sigma^{-1}"swap]\ar[rr,"F\times\prod_j F"]&&|[alias=N2]|\Nn(\la Fb\ra;Fc)\times\prod_{j=1}^n\Nn(\langle Fa_j\rangle;Fb_j)\ar[d,"\sigma\times 
            \sigma^{-1}"]\\
            |[alias=M2]|\M (\la b\ra\sigma ;c)\times\prod_{j=1}^n\M(\langle a_{\sigma(j)}\rangle;b_{\sigma(j)})\ar[rr,"F\times\prod_j F"swap]\ar[d,"\gamma"{swap,name=e}]&&\Nn(\la Fb\ra\sigma;Fc)\times\prod_{j=1}^n\Nn(\langle Fa_{\sigma(j)}\rangle;Fb_{\sigma(j)})\ar[d,"\gamma"name=f]\\
            \M(\la\la a_{\sigma(j)}\ra\ra_{j=1}^n;c)\ar[rr,"F"swap]&&\Nn(\la\la Fa_{\sigma(j)}\ra\ra_{j=1}^n;Fc)
            \arrow[Rightarrow,from=M2, to=N2,shorten=12mm,"F_{\sigma}\times 1"swap]
            \arrow[Rightarrow,from=M1, to=N1,shorten=14mm,"F_{\sigma\la\id_{k_j}\ra}"swap]
        \end{tikzcd}
\end{center}
Here $\sigma\la \id_{k_{\sigma(j)}}\ra=\sigma\la\id_{k_{\sigma(1)}},\dots,\id_{k_\sigma(n)}\ra.$ This means that for 1-cells $f\in\ob(\M(\la b\ra;c))$ and $g_j\in \ob(\M(\la a_j\ra;b_j))$ for $1\leq j\leq n,$
\begin{equation}\label{4.1.11}
F_{\sigma\la \id_{k_{\sigma(j)}}\ra;\gamma(f;\la g_j\ra)}=\gamma\left(F_{\sigma;f};\la 1_{Fg_{\sigma (j)}}\ra_{j=1}^n\right).
\end{equation}
The domains and codomains of these pasting diagrams are equal by top equivariance in $\M$ and $\N,$ and the fact that $F$ preserves $\gamma$ implies the commutativity of the empty rectangles, see \cite{Y23}.\\

\item \textbf{Bottom Equivariance:} For every $c\in\ob(\M),$ $n\geq 0,$ $\la b \ra =\la b_j\ra_{j=1}^n\in \ob(M)^n,$ $k_j\geq 0$ for $1\leq j\leq n,$   and $\la a_j\ra=\la a_{j,i}\ra_{i=1}^{k_j}\in \ob(\M)^{k_j}$ for $1\leq j \leq n$ and $1\leq i \leq k_j,$ and $\tau_j \in \Sigma_{k_j},$ the following two pasting diagrams are equal.
\begin{center}
        \begin{tikzcd}[font=\small]
            \M (\la b\ra;c)\times\prod_{j=1}^n\M(\langle a_j\rangle;b_j)\ar[rr,"F\times\prod_j F"]\ar[d,"\gamma"{swap,name=h}]&[-10pt]&[-30pt]\Nn(\la Fb\ra;Fc)\times\prod_{j=1}^n(\langle Fa_j\rangle;Fb_j)\ar[d,"\gamma"name=g]\\
            
            \M(\la\la a_j\ra\ra_{j=1}^n;c)\ar[d,"\id_n\la\tau_j\ra"swap]\ar[rr,"F"]&&|[alias=N1]|\Nn(\la\la Fa_j\ra\ra_{j=1}^n;Fc)\ar[d,"\id_n\la\tau_j\ra"]\\
            |[alias=M1]|\M(\la\la a_j\ra\tau_j\ra_{j=1}^n;c)\ar[rr,"F"swap]&&\Nn(\la\la Fa_j\ra\tau_j\ra_{j=1}^n;Fc)\\[-10pt]
            
            &\| &\\[-10pt]
            \M (\la b\ra;c)\times\prod_{j=1}^n\M(\langle a_j\rangle;b_j)\ar[d,"\id\times \prod_j\tau_j"swap]\ar[rr,"F\times\prod_j F"]&&|[alias=N2]|\Nn(\la Fb\ra;Fc)\times\prod_{j=1}^n\Nn(\langle Fa_j\rangle;Fb_j)\ar[d,"\id\times \prod_j\tau_j"]\\
            |[alias=M2]|\M (\la b\ra;c)\times\prod_{j=1}^n\M(\langle a_j\rangle \tau_j;b_j)\ar[rr,"F\times\prod_j F"swap]\ar[d,"\gamma"{swap,name=e}]&&\Nn(\la Fb\ra;Fc)\times\prod_{j=1}^n\Nn(\la\langle Fa_j\rangle\tau_j\ra;Fb_j)\ar[d,"\gamma"name=f]\\
            \M(\la\la a_j\ra\tau_j\ra_{j=1}^n;c)\ar[rr,"F"swap]&&\Nn(\la\la Fa_j\ra\tau_j\ra_{j=1}^n;Fc)
            \arrow[Rightarrow,from=M1, to=N1,shorten=13mm,"F_{\id_n\la\tau_i\ra}"swap]
            \arrow[Rightarrow,from=M2, to=N2,shorten=12mm,"1\times\prod_j F_{\tau_j}"swap]
        \end{tikzcd}
\end{center}
This means that for 1-cells $f\colon \la b\ra\to c$ and $g_j\colon \la a_j\ra\to b_j$ for $1\leq j\leq n,$
\begin{equation}\label{4.1.12}
    F_{\id_n\la\tau_j\ra;\gamma(f;\la g_j\ra})=\gamma(1_{Ff};\la F_{\tau_j;g_j}\ra)
\end{equation}
as 2-cells in $\Nn(\la\la Fa_j\ra\tau_j\ra;Fc).$ The domain and codomain of these pasting diagrams are equal by bottom equivariance for $\M$ and $\N$, and the preservation of $\gamma$ by $F$ guarantees that the empty squares commute, see \cite{Y23}.
\end{itemize}
\end{defi}
Next we describe the horizontal composition of 1-cells in the 2-category $\Cat$-$\mathbf{Multicat^{ps}}.$ 
\begin{defi}\normalfont{(\cite{Y23} Def. 4.1.1)}
 Let $F\colon \M\to\Nn,$ and $G\colon \Nn\to \mathcal{Q}$ be pseudo symmetric $\Cat$-multifunctors. We define the pseudo symmetric functor $GF\colon \M\to \mathcal{Q}.$ On objects $GF$ is the composite function $GF\colon \ob(\M)\to\ob(\mathcal{Q}).$ The composite component functor is given for $\la a\ra\in\ob(\M)^n,$ and $b\in\ob(\M)$ by the pasting
\begin{center}
    \begin{tikzcd}
        \M(\la a\ra;b)\ar[r,"F"]&\Nn(\la Fa\ra;b)\ar[r,"G"]&\Qq (\la GFa\ra;GF b).
    \end{tikzcd}
\end{center}
The symmetry isomorphisms are given for each $\sigma\in\Sigma_n,$ $\la a\ra \in\ob (\M),$ and $b\in \ob(\M)$ by
\begin{center}
    \begin{tikzcd}
        \M(\la a\ra;b)\ar[d,"\sigma"swap]\ar[r,"F"]&|[alias=N1]|\Nn(\la Fa\ra;Fb)\ar[d,"\sigma"]\ar[r,"G"]&|[alias=Q]|\Qq(\la GF a\ra;GFb)\ar[d,"\sigma"]\\
        |[alias=M]|\M(\la a\ra\sigma;b)\ar[r,"F"]&|[alias=N2]|\Nn(\la Fa\ra\sigma;fb)\ar[r,"G"]&\Qq(\la GFa\ra\sigma;GFb).
        \arrow[Rightarrow,from=M, to=N1,shorten=4mm,"F_{\sigma,\la a\ra,b}"swap]
        \arrow[Rightarrow,from=N2, to=Q,shorten=5mm,"G_{\sigma,\la Fa\ra,Fb}"swap]
    \end{tikzcd}
\end{center}
That is, for each 1-cell $f\colon \la a\ra\to b,$ the $f$ component of $GF_\sigma$ is given by the composite
\begin{equation}\label{horcomp}
\begin{tikzcd}
&G((Ff)\sigma)\ar[rd,"G_{\sigma;Ff}"]&\\
    GF(f\sigma)\ar[rr,"(GF)_{\sigma;f}"swap]\ar[ru,"G(F_{\sigma;f})"]&&(GFf)\sigma .
\end{tikzcd}
\end{equation}
\end{defi}
Next we define the 2-cells of the category $\Cat$-$\mathbf{Multicat^{ps}}.$
\begin{defi}\normalfont{(\cite{Y23} Def. 4.2.1)} 
Suppose that $F,G\colon \M\to \Nn$ are pseudo symmetric $\Cat$-multifunctors. A \textit{pseudo symmetric} $\Cat$-multinatural transformation $\theta\colon F\Rightarrow G$ is the data of a component 1-cell $\theta_a:Fa\to Ga$ for each $a\in\ob(\M)$ subject to axioms \Cref{trans1}, \Cref{trans2} and the following extra axiom. For each $n\geq 0,$ $\la a\ra\in\ob(\M)^n,$ $b\in\ob(M),$  object $f\in \ob(\M(\la a\ra;b)),$ and permutation $\sigma\in \Sigma_n,$ the following arrow equality holds in the category $\Nn(\la Fa\ra\sigma;b),$ 
\begin{center}
    \begin{equation}
    \gamma\left(1_{\theta_{b}};F_{\sigma;f}\right)=\gamma\left( G_{\sigma;f};\la 1_{\theta_{ a_{\sigma(j)}}}\ra\right).
    \end{equation}
\end{center}
This can also be expressed diagrammatically as the equality of multicategorical pasting diagrams 
\begin{center}
           \begin{center}
           \begin{tikzcd}
\la Fa\ra \sigma\arrow[dd, bend left=50, "(Ff)\sigma"{name=U}]\arrow[dd, bend right=50, "F(f\sigma)"'{name=D}]\ar[r,"\langle \theta_{a_{\sigma(j)}}\rangle"]&G\la a\ra\sigma\ar[dd,"(Gf)\sigma"]&& \la Fa\ra\sigma\ar[r,"\la\theta_{a_{\sigma(j)}}\ra"]\ar[dd,"F(f\sigma)"']&\la Ga\ra\sigma\arrow[dd, bend left=50, "(Gf)\sigma"{name=A}]\arrow[dd, bend right=50, "G(f\sigma)"'{name=B}]\\[-12pt]
&&\mathbf{=}&&\\[-12pt]
Fb\ar[r,"\theta_b"']&Gb&&Fb\ar[r,"\theta_b"]&Gb,
\arrow[Rightarrow, from=D, to=U,shorten <=2mm,shorten >=1mm,"F_{\sigma;f}"]
\arrow[Rightarrow, from=B, to=A,shorten <=2mm,shorten >=1mm,"G_{\sigma;f}"]
\end{tikzcd}
    \end{center}
    \end{center}
where the diagrams are interpreted as in \Cref{gammanotation}, the squares commuting by \Cref{trans1} and top and bottom equivariance for $\N,$ see \cite{Y23}.\\

We define the vertical and horizontal composition of pseudo symmetric $\Cat$-multinatural transformations in the same way that we did for symmetric ones, through diagrams \Cref{vert} and \Cref{hor}. 
\end{defi}
It is a theorem of Yau \cite{Y23} that the data we have just defined gives the structure of a 2-category, namely $\Cat$-$\mathbf{Multicat^{ps}}.$ \Cref{2D} says that we can describe this 2-category solely in terms of symmetric $\Cat$-multifunctors and symmetric $\Cat$-multinatural transformations. 

\section{Equivalent definition of Pseudo Symmetry}\label{Section3}
To prove our first result we use finite products in the category $\Cat\mbox{-}\textbf{Multicat}.$ Having just the 1-categorical structure in mind, the products in $\Cat\mbox{-}\textbf{Multicat}$ are given in the following way. If $\mathcal{M}$ and $\mathcal{N}$ are two $\Cat$-multicategories, then $\mathcal{M}\times \mathcal{N}$ has objects $\ob(\mathcal{M}\times \mathcal{N})=\ob(\mathcal{M})\times \ob(\mathcal{N}).$ Now, for $n\geq 0,$ $\la a\ra\in\ob(\M)^n,$ $\la c\ra\in\ob(\Nn)^n,$ $b\in \ob(\M),$ and $d\in\ob(\Nn),$ we define 
$$\mathcal{M}\times\mathcal{N}(\langle (a,c)\rangle;(b,d))=\mathcal{M}(\langle a\rangle ;b)\times \mathcal{N}(\langle c\rangle;d).$$

The composition $\gamma$ of $\mathcal{M}\times\mathcal{N},$ as well as the $\Sigma_*$ action and the multicategorical units, are defined componentwise. Next we define the pseudo symmetric multifunctor $\eta_\mathcal{M}$ appearing in the statement of \ref{mainresult}.
\begin{defi}
    Let $\M$ be a $\Cat$-multicategory. We define the pseudo symmetric $\Cat$-multifunctor $\eta_\M\colon \M\to\M\times E\Sigma_*$ which, when there is no room for confusion, we will denote $\eta.$ For an object $a\in\ob(\M)$ as $\eta(a)=(a,*).$ We will abuse notation and denote the object $(a,*)$ of $\M\times E\Sigma_*$ as $a.$ \\

For  $n\geq 0,$ $\langle a\rangle\in \ob(\mathcal{M})^n$ and $b\in \ob(\mathcal{M})$ we need to define a functor $\eta\colon\mathcal{M}(\langle a\rangle;b) \to \mathcal{M}(\langle a\rangle;b)\times E\Sigma_n$. For a 1-cell $f\colon \la a\ra\to b,$ we define 
$$\eta(f)=(f,\text{id}_n)\in\ob(\M (a;b)\times E\Sigma_n).$$
Similarly, for a 2-cell $\alpha\colon f\to g$ in $\mathcal{M}(\langle a\rangle;b),$
$$\eta(\alpha)=(\alpha,1_{\text{id}_n})\in\M(\la a\ra;b)\times E\Sigma_n((f,\id_n)),(g,\id_n).$$

Next, we define the components of the pseudo symmetry isomorphisms. For $\sigma,\tau\in \Sigma_n$ we will denote from here on by $E_\sigma^\tau$ the unique arrow $\sigma\to \tau$ in $E\Sigma_n.$ For $\sigma\in \Sigma_n,$ $\langle a\rangle\in \ob(\M)^n,$ and $b\in \ob(\M)$ we need to define a natural isomorphism $\eta_{\sigma,{\langle a\rangle},b}\colon (\eta\circ\sigma)\to (\sigma\circ \eta)$ that fits in the following diagram
\begin{center}
\begin{tikzcd}
  \mathcal{M}(\langle a\rangle;b)\ar[r,"\eta"]\ar[d,"\sigma"swap]&|[alias=N]|\mathcal{M}(\langle a\rangle;b)\times E\Sigma_n\ar[d,"\sigma\times\sigma"]\\
  |[alias=M]|\mathcal{M}(\langle a\rangle \sigma,b)\ar[r,"\eta"swap]&\mathcal{M}(\langle a\rangle \sigma;b)\times E\Sigma_n.
  \arrow[Rightarrow,from=M, to=N,shorten >=4mm,shorten <=4mm,"\eta_{\sigma,\langle a\rangle,b}"swap,"\cong"]
\end{tikzcd}
\end{center}

 The isomorphism $\eta_{\sigma,{\langle a\rangle},b}$ is defined for every 1-cell $f\colon\la a\ra\to b$ as the 2-cell $$\eta_{\sigma;f}=(1_{f\sigma},E_{\text{id}}^{\sigma})\colon(f\sigma,\text{id}_n)\to (f\sigma,\sigma).$$  
\end{defi}
\begin{lem}
    Let $\M$ be a $\Cat$-multicategory, then $\eta_\M\colon \M\to \M\times E\Sigma_*$ is pseudo symmetric.
\end{lem}
\begin{proof}
We start from a non symmetric multifunctor $\eta\colon \M\to\M\times E\Sigma_*$ that is the identity on the first coordinate and the multicategorical unit in the second coordinate. As a non symmetric multifunctor, $\eta$ preserves units and $\gamma$ composition. We need to show that $\eta$ is a pseudo symmetric $\textbf{Cat}$-multifunctor. The naturality of $\eta_{\sigma;f}$ follows from the commutativity of the following diagram for any 2-cell $\alpha\colon f\to g$: 
 \begin{center}
\begin{tabular}{c}
    \xymatrix{}
\end{tabular}
\end{center}
\begin{center}
    \begin{tikzcd}
        (f\sigma,\text{id}_n)\ar[r,"(1_{f\sigma}{,}E_{\text{id}_n}^{\sigma})"]\ar[d,"(\alpha\sigma{,}1_{\text{id}_n})"swap]&[10pt](f\sigma,\sigma)\ar[d,"(\alpha\sigma{,}1_{\sigma})"]\\(g\sigma,\text{id}_n)\ar[r,"(1_{g\sigma}{,}E_{\text{id}_n}^\sigma)"swap]&(g\sigma,\sigma).
    \end{tikzcd}
\end{center}
Next we focus on the coherence axioms. The unit permutation axiom \Cref{pseudounit} holds since,  for all $\la a\ra\in \ob(\M)^n,$ $b\in\ob(\M),$ and $f\colon \la a\ra\to b,$ $$\eta_{\text{id}_n;f}=(1_{f\id_n},E_{\id_n}^{\id_n})=(1_f,1_{\id_n})=1_{(f,\id_n)}=1_{\eta(f)}.$$

Let $\la a\ra,b$ and $f$ be as before, the product permutation axiom \Cref{4.1.10} holds again by definition. Indeed, for $\tau,\sigma\in \Sigma_n,$  we have
$$\eta_{\sigma\tau;f}=(1_{f\sigma\tau},E_{\text{id}}^{\sigma\tau})=(1_{f\sigma\tau},E_{\tau}^{\sigma\tau})\circ(1_{f\sigma\tau},E_{\id_n}^\tau)=(\eta_{\sigma;f}\tau )\circ \eta_{\tau;f\sigma}.$$ 
For Top Equivariance \Cref{4.1.11}, suppose that  $c\in\ob(\M),$ $n\geq 1,$ $\la b \ra =\la b_j\ra_{j=1}^n\in\ob(\M)^n,$ $k_j\geq 0$ for $1\leq j\leq n,$ $\la a_j\ra=\la a_{j,i}\ra_{i=1}^{k_j}\in \ob(\M)^{k_j}$ for $1\leq j \leq n,$  $\sigma\in \Sigma_n,$ $f\in\ob(\M(\la b\ra;c)),$ and $g_j\in\ob(\M(\la a_j\ra;b_j)).$ We have that
\begin{align*}
    \gamma(\eta_{\sigma;f};\langle 1_{i(g_{\sigma(j)})}\rangle)&=\gamma((1_{f\sigma},E_{\text{id}}^\sigma);\langle (1_{g_{\sigma(j)}},1_{\text{id}_{k_{\sigma(j)}}})\rangle)\\
    &=\left((\gamma(1_{f\sigma};1_{g_{\sigma(j)}}),\gamma\left(E_{\text{id}}^\sigma;E_{\text{id}_{k_{\sigma(j)}}}^{\text{id}_{k_{\sigma(j)}}}\right)\right)\\
    &=\left(1_{\gamma(f;\langle g_{\sigma(j)}\rangle)},E_{\text{id}\langle\text{id}_{k_{\sigma(j)}}\rangle}^{\sigma\langle \text{id}_{k_{\sigma (j)}}\rangle}\right)\\
    &=\left(1_{\gamma(f;\langle g_j\rangle)\sigma\langle\text{id}_{k_{\sigma(j)}}\rangle},E_{\text{id}_k}^{\sigma\langle\text{id}_{k_{\sigma(j)}}\rangle}\right)\\
    &=\eta_{\sigma\langle\text{id}_{k_{\sigma(j)}}\rangle;\gamma(f;\langle g_j\rangle)}.
\end{align*}
For Bottom Equivariance, let $c,$ $n,$ $\la b \ra ,$ $k_j$ for $1\leq j\leq n,$ $\la a_j\ra$ for $1\leq j \leq n,$  $f$ and $g_j$ be as above and let $\tau_j\in \Sigma_{k_j}$ for $1\leq j \leq n.$ We also let $k=\sum_{j=1}^{n}k_j.$ Bottom Equivariance \Cref{4.1.12} for $i$ is 
\begin{align*}
    \gamma\left(1_{if};\langle \eta_{\tau_j;g_j}\rangle\right)&=\gamma\left((1_f,1_{\text{id}_n});\langle (1_{g_j\tau_j},E_{\text{id}_{k_j}}^{\tau_j})\rangle\right)\\
    &=\left(\gamma(1_f;1_{g_j\tau_j}),1_{\text{id}_n}\langle E_{\text{id}_{k_j}}^{\tau_j}\rangle\right)\\
    &=\left(1_{\gamma(f;\langle g_j\tau_j\rangle)},E_{\text{id}_k}^{\text{id}_n \langle \tau_j\rangle}\right)\\
    &=\left(1_{\gamma(f;\langle g_j\rangle)\text{id}_n\langle \tau_j\rangle},E_{\text{id}_k}^{\text{id}_n \langle \tau_j\rangle}\right)\\
    &=\eta_{\text{id}\langle\tau_j\rangle,\gamma (f;\langle g_j\rangle)}.
\end{align*}
Thus, we conclude that $\eta\colon\mathcal{M}\to \mathcal{M}\times E\Sigma_*$ is a pseudo symmetric $\textbf{Cat}$-multifunctor.
\end{proof}
Recall that $j\colon \Cat\mbox{-}\mathbf{Multicat}\to\Cat\mbox{-}\mathbf{Multicat^{ps}}$ denotes the inclusion functor. We are ready to present a proof of \ref{mainresult}. 
\begin{theorem}\label{mainprop} Let $\mathcal{M}$ and $\N$ be a $\textbf{Cat}$-multicategories and $F\colon \mathcal{M}\rightarrow \mathcal{N}$ a pseudo symmetric $\Cat$-multifunctor. There exists a unique symmetric $\Cat$-multifunctor $\phi(F)\colon \mathcal{M}\times E\Sigma_*\rightarrow \mathcal{N}$ such that the following diagram commutes: 
\begin{center}    
\begin{tikzcd}
        &\M\times E\Sigma_*\ar[rd,"j\phi(F)"]&\\
        \M\ar[rr,"F"swap]\ar[ru,"\eta_\M"]&&\Nn.\\
    \end{tikzcd}
\end{center} That is, $F=j\phi(F)\circ \eta_\M$ in $\Cat\mbox{-}\mathbf{Multicat^{ps}}$.
\end{theorem}
\begin{proof}[Proof of \Cref{mainresult}]
For uniqueness, suppose that  $\phi(F)\colon \mathcal{M}\times E\Sigma_*\to \mathcal{N}$ is a symmetric $\Cat$-multifunctor satisfying $F=(j\phi(F))\circ \eta$. We will abuse notation and write $j\phi(F)=\phi(F).$ We will prove there is a unique way of defining $\phi(F).$ At the level of the objects of the multicategory we must have  $\phi(F)(a,*)=\phi(F)\circ \eta(a)=F(a)$ for each $a\in\ob(\mathcal{M}).$  Next, we show that there is a unique way of defining each component functor of $\phi(F).$ For this let  $\langle a\rangle\in \ob(\mathcal{M})^n,$ $b\in \ob(\mathcal{M}),$ and consider the functor  $\phi(F)\colon \mathcal{M}(\langle a\rangle;b)\times E\Sigma_n\to \mathcal{N}(\langle Fa\rangle;Fb).$ If  $f\colon\la a\ra\to b$ is an 1-cell and $\sigma\in \Sigma_n,$ we must have that
\begin{align}\label{defob}
    \phi(F)(f,\sigma)&=\phi(F)((f\sigma^{-1},\text{id}_n)\sigma)\nonumber\\
    &=\phi(F)((f\sigma^{-1},\text{id}_n))\sigma\nonumber\\
    &=\phi(F)\circ \eta(f\sigma^{-1})\sigma\nonumber\\
    &=F(f\sigma^{-1})\sigma,
\end{align}
where in the second equality we used that $\phi(F)$ is symmetric. So the values of the component functors of $\phi(F)$ on $n$-ary 1-cells are uniquely determined by $F$. In exactly the same fashion, for $\la a\ra,b$ and $\sigma$  as before, $f,g\colon \la a\ra\to b,$ and $\alpha\colon f\to g$ a 2-cell, 
\begin{equation}\label{def1st}
    \phi(F)(\alpha,1_\sigma)=F(\alpha\sigma^{-1})\sigma.
\end{equation}
Finally, if $f,\sigma$ are as before and  $\tau\in \Sigma_n$, we get that
\begin{align}\label{def2nd}
    \phi(F)(1_f,E_{\sigma}^\tau)&=\phi(F)(1_{f\sigma^{-1}}\sigma,E_{\text{id}}^{\tau\sigma^{-1}}\sigma)\nonumber\\
    &= \phi(F)((1_{f\sigma^{-1}},E_{\text{id}}^{\tau\sigma^{-1}}))\sigma\nonumber\\
    &=\phi(F)(\eta_{\tau\sigma^{-1};f\tau^{-1}})\sigma\nonumber\\
    &=(\phi(F)\circ \eta_{\tau\sigma^{-1};f\tau^{-1}})\sigma\nonumber\\
    &=(F_{\tau\sigma^{-1};f\tau^{-1}})\sigma.
\end{align}
We have used the definition of composition of pseudo symmetric $\Cat$-multifunctors \Cref{horcomp} where we see $\phi(F)$ trivially as a pseudo symmetric functor. Since, for $\la a\ra,b,f,g,$ $\alpha,\sigma,$ and $\tau$ as before we can write the morphism $(\alpha\colon f\to g,E_{\sigma}^\tau)$ in $\mathcal{M}(\langle a\rangle;b)\times \Sigma_n$ as $ (1_y,E_\sigma ^\tau)\circ (f,1_{\sigma})$ for and both $\phi(F)(1_y,E_\sigma^\tau)$ and $\phi(F)(f,1_\sigma)$ are uniquely determined by $F$ we conclude that the component functors of $\phi(F)$ are uniquely determined by $F,$ and so we have proven the uniqueness of $\phi(F).$ \\

 Next we prove the existence of $\phi(F)$. By uniqueness, we have no choice but to define $\phi(F)(b,*)=Fb$ for any $b\in\ob(M).$ Likewise, for $\la a\ra\in\ob(\M)^n$ and $b\in \ob(\M),$ uniqueness forces the definition  of the component functor  $\phi(F)\colon \M(\la a\ra;b)\times \Sigma_n\to \N(\la Fa\ra;b).$ For $f\colon \la a\ra\to b$ and $\sigma\in\Sigma_n$ $\phi(F)(f,\sigma)$ is defined by \Cref{defob}, for  $\alpha\colon f\to g$ a 2-cell, we define $\phi(F)(\alpha,1_\sigma)$ by \Cref{def1st} and for $\tau \in \Sigma_n$ we define $\phi(F)(1_f,E_\sigma^\tau)$ by \Cref{def2nd}. First, we notice that for a 1-cell $f\colon \la a\ra\to b$ such definition is ambiguous for the identity arrow $(1_f,1_\sigma)$ since both \Cref{def1st} and \Cref{def2nd} apply. However, $\phi(F)$ is well defined in this case since $F$ is a functor componentwise and so it preserves identities. Explicitly,
 
$$F(1_f\sigma^{-1})\sigma=F(1_{f\sigma^{-1}})\sigma=1_{F(f\sigma^{-1})}\sigma=1_{F(f\sigma^{-1})\sigma},$$

and 

$$(F_{\sigma\sigma^{-1},f\sigma^{-1}})\sigma=F_{\text{id}_n,f\sigma^{-1}}\sigma=1_{F(f\sigma^{-1})}\sigma=1_{F(f\sigma^{-1})\sigma}.$$

So our definition is so far unambiguous and $\phi(F)$ preserves identities. Next, we go on to extend the definition of $\phi(F)$ to the rest of the arrows. For $\alpha\colon f\to g$ 2-cell in $\mathcal{M}(\langle a\rangle,b)$  and $\sigma,\tau$ in $\Sigma_n,$ we define $\phi(F)(\alpha,E_\sigma^\tau)\colon F(f\sigma^{-1})\sigma\to F(g\tau^{-1})\tau$ by 
\begin{align}\label{sq}
    \phi(F)(\alpha,E_\sigma^\tau)=&\phi(F)(1_g,E_\sigma^\tau)\circ \phi(F)(\alpha,1_\sigma)\nonumber\\
    =&\phi(F)(\alpha,1_\tau)\circ \phi(F)(1_f,E_\sigma^\tau).
\end{align}
The last equality together with the preservation of identities already proven implies that our definition is unambiguous. This equality holds since,
\begin{align*}
\phi(F)(1_g,E_\sigma^\tau)\circ \phi(F)(\alpha,1_\sigma)&=\left(F_{\tau\sigma^{-1};g\tau^{-1}}\right)\sigma \circ F(\alpha \sigma^{-1})\sigma\\
&=\left(F_{\tau\sigma^{-1};g\tau^{-1}} \circ F(\alpha \sigma^{-1})\right)\sigma\\
&=\left(F(\alpha\tau^{-1})\tau\sigma^{-1} \circ F_{\tau\sigma^{-1};f\tau^{-1}}\right)\sigma\\
&=F(\alpha\tau^{-1})\tau \circ \left(F_{\tau\sigma^{-1};f\tau^{-1}}\right)\sigma\\
&=\phi(F)(\alpha,1_\tau)\circ \phi(F)(1_f,E_\sigma^\tau).
\end{align*}
The third equality holds since the commutativity of the following diagram is an instance of the pseudo symmetry naturality coherence axiom for $F$ \Cref{4.1.9}. Explicitly, 
\begin{equation}
\begin{tikzcd}
F(f\tau^{-1}\tau\sigma^{-1})\ar[r,"F_{\tau\sigma^{-1};f\tau^{-1}}"]\ar[d,"F(\alpha \tau^{-1}\tau\sigma^{-1})"swap]&[30pt]F(f\tau^{-1})\tau\sigma^{-1}\ar[d,"(F\alpha\tau^{-1})\tau\sigma^{-1}"]\\
F(g\tau^{-1}\tau\sigma^{-1})\ar[r,"F_{\tau\sigma^{-1};g\tau^{-1}}"swap]&F(g\tau^{-1})\tau\sigma^{-1}.
    \end{tikzcd}
\end{equation}
Next, we check that the defined assignments give a functor $\phi(F)\colon \M(\la a\ra;b)\times E\Sigma_n\to \mathcal{N}(\la Fa\ra;b).$  The fact that $\phi(F)$ preserves identities was already proven. We prove functoriality in the first variable first.  For $f\colon\la a\ra\to b$  1-cell, $\sigma,\tau ,$ and $\rho$ in $\Sigma_n$,  \begin{align}\label{tr1}
   \phi(F)(1_f,E_\tau^\rho)\circ \phi(F)(1_f,E_\sigma^\tau)&=\left(F_{\rho\tau^{-1};f\rho^{-1}}\tau\right) \circ \left(F_{\tau\sigma^{-1};f\tau^{-1}}\sigma\right)\nonumber\\
    &=\left(\left(F_{\rho\tau^{-1};f\rho^{-1}}\right)\tau\sigma^{-1} \circ F_{\tau\sigma^{-1};f\tau^{-1}}\right)\sigma\nonumber\\
    &=\left(F_{\rho\sigma^{-1};f\rho^{-1}}\right)\sigma\nonumber\\
    &=\phi(F)(1_f,E_\sigma^\rho).
\end{align}
Here the third equality holds by \Cref{4.1.10}, which implies the commutativity of the following diagram:
\begin{equation}\label{product}
       \begin{tikzcd}
       &F(f\rho^{-1}\rho\tau^{-1})\tau\sigma^{-1}\ar[rd,"(F_{\rho\tau^{-1};f\rho^{-1}})\tau\sigma^{-1}"]&\\
F(f\rho^{-1}\rho\tau^{-1}\tau\sigma^{-1})\ar[ru,"F_{\tau\sigma^{-1};f\rho^{-1}\rho\tau^{-1}}"]\ar[rr,"F_{\rho\tau^{-1}\tau\sigma^{-1};f\rho^{-1}}"']&&F(f\rho^{-1})\rho\tau^{-1}\tau\sigma^{-1}.
    \end{tikzcd} 
    \end{equation}
On the other hand, if $\alpha\colon f\to g$ and $\beta\colon g\to h$ are 2-cells in $\mathcal{M}(\langle a\rangle;b),$ and $\sigma\in \Sigma_n$ we have that
\begin{align}\label{tr2}
    \phi(F)(\beta,1_\sigma)\circ \phi(F)(\alpha,1_\sigma)=\phi(F)(\beta\alpha,1_\sigma).
\end{align}
The functoriality of $\phi(F)$ follows from a straightforward argument by  \cref{tr1,tr2} together with the exchange property \Cref{sq}.\\

The next step is to prove that the component functors give rise to a symmetric $\textbf{Cat}$-multifunctor $\phi(F)\colon \mathcal{M}\times E\Sigma_*\to \mathcal{N}.$ First, notice that $\phi(F)$ preserves units since, for $a\in \ob(\mathcal{M})$ $\phi(F)(1_a,\text{id}_1)=F(1_a\text{id}_1^{-1})\text{id}_1=F(1_a)=1_{Fa},$ since $F$ itself preserves units.  Next we prove that $\phi(F)$ preserves the $\Sigma_n$-action. For $n\geq 0,$  $\la a\ra\in\ob(\M)^n,$ $b\in \ob(\M),$ and $\sigma\in \Sigma_n$,  we show that the following diagram commutes in $\textbf{Cat}:$
\begin{center}
\begin{tabular}{c}
    \xymatrix{\mathcal{M}(\langle a\rangle;b)\times E\Sigma_n\ar^{\phi(F)}[r]\ar[d]_{\sigma}&\mathcal{N}(\langle Fa\rangle;Fb)\ar[d]^{\sigma}\\
   \mathcal{M}(\langle a_j\rangle\sigma;b)\times E\Sigma_n\ar[r]_{\phi(F)}&\mathcal{N}(\langle Fa\rangle\sigma;Fb).}\\
\end{tabular}
\end{center}
For this we don't need any of the pseudo symmetry axioms for $F.$ For 1-cells $(f\colon \la a\ra\to b,\tau)$ of $\mathcal{M}(\langle a\rangle;b)\times E\Sigma_n,$
\begin{align*}
    \phi(F)(f,\tau)\sigma&=(F(f\tau^{-1})\tau)\sigma\\
    &=F(f\tau^{-1})\tau\sigma\\
    &=F(f\sigma(\tau \sigma)^{-1})\tau\sigma\\
    &=\phi(F)((f\sigma,\tau\sigma)))\\
    &=\phi(F)((f,\tau)\sigma).
\end{align*}
A similar calculation works for 2-cells of the form $(\alpha\colon f\to g,1_\tau)$ in $ \mathcal{M}(\langle a\rangle;b)\times E\Sigma_n.$
For morphisms of the form $(1_f,E_\tau^\rho)$ in $\mathcal{M}(\langle a\rangle;b)\times E\Sigma_n,$
\begin{align*}
    (\phi(F)(1_f,E_\tau^\rho))\sigma&=(F_{\rho\tau^{-1};f\rho^{-1}}\tau)\sigma\\
    &=F_{\rho\tau^{-1};f\rho^{-1}}(\tau\sigma)\\
    &=F_{\rho\sigma(\tau\sigma)^{-1};f\sigma(\rho\sigma)^{-1}}(\tau\sigma)\\
    &=\phi(F)(1_{f\sigma},E_{\tau\sigma}^{\rho\sigma})\\
    &=\phi(F)((1_f,E_\tau^\rho)\sigma).
\end{align*}
By functoriality of $\phi(F)$ and $\sigma$ we conclude that $\phi(F)$ preserves the action of the symmetric group. \\

The only step we are missing to finish proving that $\phi(F)$ defines a $\textbf{Cat}$-multifunctor is the preservation of $\gamma.$ Let $c\in\ob(\M),$ $n\geq 0,$ $\la b\ra\in\ob(\M)^{n},$ $k_j\geq 0$ for $1\leq j\leq n,$ $\la a_j\ra=\la a_{j,i}\ra_{i=1}^{k_j} $ for $1\leq j \leq n.$ Set $k=\sum_{j=1}^nk_j.$ As usual $\la a\ra=\la a_j\ra=\la\la a_{j,i}\ra_{i=1}^{k_j}\ra_{j=1}^n$ denotes the concatenation of the $a_j$'s. We will prove that the following square is commutative:

\begin{equation}\label{cuadrado}
    \begin{tikzcd}[font=\tiny]
        \mathcal{M}(\langle b\rangle;c)\times E\Sigma_n\times \prod_{j=1}^n\mathcal{M}(\langle a_{j}\rangle;b_j)\times E\Sigma_{k_j}\ar[r,"\phi(F)\times\prod \phi(F)"]\ar[d,"\gamma"]&\mathcal{N}(\langle Fb\rangle;Fc)\times\prod_{j=1}^n\mathcal{N}(\langle F a_{j}\rangle;Fb_j)\ar[d,"\gamma"]\\
    \mathcal{M}(\langle a\rangle;c)\times E(\Sigma_k)\ar[r,"\phi(F)"swap]&\mathcal{N}(\langle F a\rangle;Fc).
    \end{tikzcd}
\end{equation}

The commutativity of this diagram at the level of 1-cells will follow from top and bottom equivariance for $\mathcal{M}$ and $\Sigma_*,$ as well as the fact that $F$ preserves $\gamma.$ Let $f\colon\la b\ra\to c,$ $\sigma\in\Sigma_n,$ and $g_j\colon \la a_j\ra\to b_j$ and $\tau_j\in\Sigma_{k_j}$ for $1\leq j \leq n.$  We have that
\begin{align*}
    \gamma(\phi(F)(f,\sigma),\langle \phi(F)(g_j,\tau_j) \rangle)&=
    \gamma(F(f\sigma^{-1})\sigma,\langle F(g_j\tau_j^{-1})\tau_j\rangle)\\
    &=\gamma\left((F(f\sigma^{-1}),\Bigl\langle F\left(g_{\sigma^{-1}(j)}\tau_{\sigma^{-1}(j)}^{-1}\right)\Bigr\rangle\right)\sigma\langle \tau_j\rangle\\
    &=F\left(\gamma\left(f\sigma^{-1},\Bigl\langle g_{\sigma^{-1}(j)}\tau_{\sigma^{-1}(j)}^{-1}\Bigr\rangle\right)\right)\sigma\langle \tau_j\rangle\\
    &=F\left(\gamma(f,\langle g_{j}\rangle)(\sigma\langle \tau_j\rangle)^{-1}\right)\sigma\langle \tau_j\rangle\\
    &=\phi(F)(\gamma(f,\langle g_j\rangle),\sigma\langle \tau_j\rangle)\\
    &=\phi(F)(\gamma((f,\sigma),\langle g_j,\tau_j\rangle)).
\end{align*}
We have proven that our diagram is commutative at the level of $1$-cells. For the morphisms we will consider again morphisms that change the first variable only and morphisms that change the second variable  only separately.\\

For 2-cells that change the first variable only, the commutativity of our diagram follows in the same way as it did for 1-cells. We consider two cases for 2-cells that change the second variable. For 2-cells of the form $((1_f,E_\sigma^\tau),\la 1_{g_j},1_{\rho_j}\ra)$ where  $f\colon\langle b\rangle\to c,$ $\sigma,\tau \in \Sigma_n,$ and   $g_j\in\ob(\mathcal{M}(\langle a_j\rangle;b_j))$ and $\rho_j\in \Sigma_{k_j}$ for $1\leq j\leq n$, we have that 
\begin{align*}
    &\gamma\left(\phi(F)(1_f,E_\sigma^\tau)\Bigl\langle \phi(F)(1_{g_j},1_{\rho_j})\Bigr\rangle\right)\\
    =&\gamma\left((F_{\tau\sigma^{-1};f\tau^{-1}})\sigma,\Bigl\langle 1_{F({g_j}\rho^{-1}_j)\rho_j}\Bigr\rangle\right)\\
=&\gamma\left(F_{\tau\sigma^{-1};f\tau^{-1}},\Bigl\langle 1_{F\left(g_{\sigma^{-1}(j)}\rho^{-1}_{\sigma^{-1}(j)}\right)}\Bigr\ra\right)\sigma\langle\rho_j \rangle\\
=&F_{\tau\sigma^{-1}\bigl\langle \id_{k_{\sigma^{-1}(j)}}\bigr\rangle;\gamma\left(f\tau^{-1}\bigl\la g_{\tau^{-1}(j)}\rho^{-1}_{\tau^{-1}(j)}\bigr\ra\right)}\sigma\langle \rho_j\rangle\\
=&F_{\tau\la \rho_j\ra(\sigma\la\rho_j\ra)^{-1};\gamma(f,\la g_j\ra)(\tau\la\rho_j\ra)^{-1}}\sigma\la\rho_j\ra\\
=&\phi(F)(1_{\gamma(f,\la g_j\ra)},E_{\sigma\la \rho_j\ra}^{\tau\la\rho_j \ra})\\
=&\phi(F)(\gamma(1_f,\la 1_{g_j}\ra),\gamma(E_\sigma^\tau,\la 1_{\rho_j}\ra)).
\end{align*}
The above equalities follow from our definitions, top and bottom equivariance in $\M,\N ,$ and $E\Sigma_*$ except the third equality which follows from top equivariance for $F$  \Cref{4.1.11}. Next, let's consider two cells of the form $((1_f,1_\sigma),\la 1_{g_j},E_{\rho_j}^{\nu_j}\ra)$ where $f\colon\langle b\rangle\to c,$ $\sigma \in \Sigma_n,$ and   $g_j\in\ob(\mathcal{M}(\langle a_j\rangle;b_j))$ and $\rho_j,\nu_j\in \Sigma_{k_j}$ for $1\leq j\leq n$. We get that \\
\begin{align*}
    &\gamma\left(\phi(F)(1_f,1_\sigma),\phi(F)\bigl\la(1_{h_j},E_{\rho_j}^{\nu_j})\bigr\ra\right)\\
    =&\gamma\left(1_{F(f\sigma^{-1})\sigma},\left( F_{\nu_j\rho_j^{-1};g_j\nu_j^{-1}}\right)\rho_j\right)\\
    =&\gamma\left(1_{F(f\sigma^{-1})},\Bigl\la F_{\nu_{\sigma^{-1}(j)}\rho^{-1}_{\sigma^{-1}(j)};g_{\sigma^{-1}(j)}\nu^{-1}_{\sigma^{-1}(j)}}\Bigr\ra\right)\sigma\la \rho_j\ra\\
    =&F_{\id_n\bigl\la \nu_{\sigma^{-1}(j)}\rho^{-1}_{\sigma^{-1}(j)}\bigr\ra;\gamma\left(f\sigma^{-1},\bigl\la g_{\sigma^{-1}(j)}\nu^{-1}_{\sigma^{-1}(j)}\bigr\ra\right)}\sigma\la \rho_j\ra\\
    =&F_{\sigma\la \nu_j\ra(\sigma\la \rho_j\ra)^{-1};\gamma(f,\la g_j\ra)(\rho\la \nu_j\ra)^{-1}}\sigma\la \rho_j\ra\\
    =&\phi(F)\left(1_{\gamma(f,\la g_j\ra)},\Bigl\la E_{\sigma\la \rho_j\ra}^{\sigma\la \nu_j\ra}\Bigr\ra\right)\\
    =&\phi(F)\left(\gamma\left((1_f,1_\sigma),\Bigl\la \left(1_{g_j},E_{\rho_j}^{\nu_j}\right)\Bigr\ra\right) \right).
\end{align*}
The third equality above follows from the bottom equivariance axiom for $F$ \Cref{4.1.12} and the rest by our definitions as well as top and bottom equivariance for $\M,\Nn ,$ and $E\Sigma_*.$\\

By functoriality of $\gamma$ and $\phi(F)$, and since every morphism in the source category can be written as a composite of arrows for which we already proved the commutativity of \Cref{cuadrado}, we can conclude that the square \Cref{cuadrado} is commutative.\\

We are almost done, we just have to prove that our definition of $\phi(F)$ gives us $F=\phi(F)\circ \eta$ in $\mathbf{Cat}\mbox{-}\mathbf{Multicat^{ps}}.$ This is clear for objects of the multicategory $\M$. For each $n\geq 0,$ $ \la a\ra\in\ob(\M)^n,$ $b\in\ob(\mathcal{M}),$ and $f\colon \la a\ra\to b,$
$$\phi(F)\circ \eta (f)=\phi(F)(f,\text{id}_n)=F(f\text{id}_n^{-1})\text{id}_n=F(f).$$ 
Similarly for $\alpha\colon f\to g$ a 2-cell in $\mathcal{M}(\langle a\rangle;b).$  Finally, we just need to prove that $(\phi(F) \circ i)_{{\sigma,\langle a_i\rangle},b}=F_{\sigma,\langle a_i\rangle},c$ for any $\sigma \in \Sigma_n.$  Let $f\colon \la a\ra\to b$ be a 1-cell. Since $\phi(F)$ is symmetric, 
$$(\phi(F) \eta)_{\sigma;f}=\phi(F)(\eta_{\sigma;f})=\phi(F)(1_{f\sigma},E_{\text{id}^\sigma})=F_{\sigma(\text{id})^{-1};f\sigma\sigma^{-1}}=F_{\sigma;f}$$
We have proven that $j\phi(F)\circ \eta=F$. This finishes our proof.
\end{proof}
Similarly, pseudo symmetric \textbf{Cat}-multinatural transformations between $F$ and $G$ correspond to symmetric \textbf{Cat}-multinatural transformations between $\phi(F)$ and $\phi(G)$.

\begin{lem}\label{2celllemma} Let $\mathcal{M},\mathcal{N}$ be $\mathbf{Cat}$-multicategories with $F,G\colon \mathcal{M}\to \mathcal{N}$ pseudo symmetric $\mathbf{Cat}$-multifunctors and $\theta\colon F\to G$ a pseudo symmetric $\Cat$-multinatural transformation. There exists a unique symmetric $\Cat$-multinatural transformation $\phi(\theta)\colon\phi(F)\to\phi(G)$ such that $\phi(\theta)* 1_{\eta_\M}= \theta$ in $\Cat\mbox{-}\mathbf{Multicat^{ps}}.$ That is, the following pasting diagram equality holds in $\Cat\mbox{-}\mathbf{Multicat^{ps}}:$
\begin{center}
    \begin{tikzcd}
        \M\arrow[rr,bend left=20,"F"{name=F}]\arrow[rr,bend right=20,"G"{swap,name=G}]\ar[rdd,"\eta_\M"']&[-15pt]&[-15pt]\Nn&[-15pt]&[-15pt]\M\ar[rr,"F"]\ar[rdd,"\eta_\M"']&[-15pt]&[-15pt]\N\\[-15pt]
        &&&=&&&\\[-15pt]
        &\M\times E\Sigma_*\ar[ruu,"\phi(G)"']&&&&\M \times E\Sigma_*. \ar[ruu,bend left=18,near start,"\phi(F)"{name=S}]\ar[ruu,bend right=22,"\phi(G)"{swap,name=L}]&
        \arrow[Rightarrow, from=F, to=G,shorten <=2mm,shorten >=1mm,"\theta"]
        \arrow[Rightarrow, from=S, to=L,shorten <=2mm,shorten >=4mm,"\phi(\theta)"']
    \end{tikzcd}
\end{center}
\end{lem}

\begin{proof}
We prove uniqueness first. Suppose $\phi(\theta)$ is a symmetric \textbf{Cat}-multinatural transformation $\phi(\theta)\colon \phi(F)\to\phi(G)$ such that $\phi(\theta)*1_{\eta}=\theta$. Any object of $\mathcal{M}\times E\Sigma_*$ takes the form $(a,*)$ for some object $a$ of $\mathcal{M},$ with $i(a)=(a,*).$  By definition,
$$\theta_a=\gamma(\phi(\theta)_{\eta a},\phi(F)((1_\eta)_a))=\gamma(\phi(\theta)_{\eta a},1_{Fa}))=\phi(\theta)_{\eta a}.$$
Since all objects of the $\Cat$-multifunctor $\mathcal{M}\times E\Sigma_*$ are of the form $\eta a$ for some object $a$ of $\mathcal{M},$ this is the only possible way of defining such \textbf{Cat}-multinatural transformation $\phi(\theta).$ Next, we check that by defining $\phi(\theta)_{(a,*)}=\theta_a$ for $a\in\ob(\M),$ we in fact get a symmetric \textbf{Cat}-multinatural transformation $\phi(\theta)\colon \phi(F)\to\phi(G).$ Let $n\geq 0,$ $\langle a\rangle\in\ob(\M)^n,$ $b\in(\ob(\M)^n),$  $f\colon\la a\ra\to b,$ and $\sigma\in \Sigma_n,$  then 
\begin{align*}
    \gamma(\phi(G)(f,\sigma);\langle \phi(\theta)_{(a_j,*)}\rangle)&=\gamma\left(G(f\sigma^{-1})\sigma;\bigl\langle \theta_{a_j}\bigr\rangle\right)\\
    &=\gamma\left(G(f\sigma^{-1});\Bigl\langle\theta_{a_{\sigma^{-1}(j)}}\Bigr\rangle\right)\sigma\\
    &=\gamma(\theta_{b};F(f\sigma^{-1}))\sigma\\
    &=\gamma(\theta_{b};F(f\sigma^{-1})\sigma)\\
    &=\gamma\left(\phi(\theta)_{(b,*)},\phi(F)(f,\sigma)\right)
\end{align*}
Where we have used top and bottom equivariance, as well as the $\Cat$-multinaturality of $\theta$. Now we need to prove $\Cat$-multinaturality of $\phi(\theta)$ for 2-cells. As before, the case where the 2-cell changes just the first variable is similar to what was done for 1-cells.
Now, if $\langle a\rangle,b,f$ are as before and $E_\sigma^\tau$ is a morphism in $E\Sigma_n,$  $(1_f,E_\sigma^\tau)$ is a morphism in $\mathcal{M}(\langle a\rangle;b)\times E\Sigma_n ,$ and
\begin{align*}
    \gamma\left(\phi(G)(1_f,E_\sigma^\tau);\Bigl\la 1_{\phi(\theta)_{(a_j,*)}}\Bigr\ra\right)&=\gamma\left((G_{\tau\sigma^{-1};f\tau^{-1}})\sigma;\langle 1_{\theta_{a_j}}\rangle\right)\\
    &=\gamma\left(G_{\tau\sigma^{-1};f\tau^{-1}};\Bigl\langle 1_{\theta_{a_{\sigma^{-1}(j)}}}\Bigr \rangle\right)\sigma\\
    &=\gamma\left(1_{\theta_{b}};F_{\tau\sigma^{-1};f\tau^{-1}}\right)\sigma\\
    &=\gamma\left( 1_{\phi(\theta)_{(b,*)}};\phi(F)(1_f,E_\sigma^\tau)\right).
\end{align*}
In the third equality we have used pseudo symmetric \textbf{Cat}-multinaturality for $\theta.$ In conclusion, by componentwise functoriality of $\gamma,\phi(F)$ and $\phi(G)$ we conclude that \textbf{Cat}-multinaturality holds for $\phi(\theta)$ at the 2-cell level finishing the proof of the lemma. 
\end{proof}
Furthermore, \Cref{mainprop} and \Cref{2celllemma} together give the following isomorphism.
\begin{cor}\label{phi} If $\mathcal{M},\mathcal{N}$ are $\mathbf{Cat}$ multicategories, then there is an isomorphism of small categories
$$\mathbf{Cat}\mbox{-}\mathbf{Multicat^{ps}}(\mathcal{M},\mathcal{N})\cong \mathbf{Cat}\mbox{-}\mathbf{Multicat}(\mathcal{M}\times E\Sigma_*,\mathcal{N}).$$
\end{cor}
\begin{proof}
Recalling the definitions from the two previous results, we define \begin{equation}\label{phimn}
    \phi\colon \mathbf{Cat}\mbox{-}\mathbf{Multicat^{ps}}(\mathcal{M},\mathcal{N})\to \mathbf{Cat}\mbox{-}\mathbf{Multicat}(\mathcal{M}\times E\Sigma_*,\mathcal{M})
\end{equation} 
for pseudo symmetric $\Cat$-multifunctors as in \Cref{mainprop} and for pseudo symmetric $\textbf{Cat}$-multinatural transformations as in \Cref{2celllemma}.\\

It is immediate from the definitions that $\phi$ is a functor. Indeed, if $\alpha\colon F\to G$ and $\beta\colon G\to H$ are pseudo symmetric \textbf{Cat}-multinatural transformations with $F,G,H\colon \mathcal{M}\to \mathcal{N}$
$$\phi(\beta * \alpha)_{(c,*)}=(\beta *\alpha)_c=\gamma(\beta_c,\alpha_c)=\gamma(\phi(\beta)_{(c,*)},\phi(\alpha)_{(c,*)})=(\phi(\beta)*\phi(\alpha))_{(c,*)}$$
We can define the inverse of $\phi,$ $\eta^*,$ as the composite
\begin{equation}\label{etamn}
    \begin{tikzcd}
        \Cat\mbox{-}\mathbf{Multicat}(\M\times E\Sigma_*,\N)\ar[rd,"\eta^*"']\ar[r,"j"]&\Cat\mbox{-}\mathbf{Multicat^{ps}}(\M\times E\Sigma_*,\N)\ar[d,"\eta_{\M}^*"]\\
        &\Cat\mbox{-}\mathbf{Multicat^{ps}}(\M,\N).
    \end{tikzcd}
\end{equation}
Finally, the existence part of \Cref{mainprop} and \Cref{2celllemma}, implies that $\eta^*\circ \phi$ is the identity of $\Cat\mbox{-}\mathbf{Multicat^{ps}}(\M,\Nn),$ while the uniqueness part of both results implies that $\phi\circ \eta^*$ is the identity of $\Cat\mbox{-}\mathbf{Multicat}(\M\times E\Sigma_*,\Nn).$
\end{proof}
The two previous results hint at the existence of a 2-adjunction between the 2-inclusion $j\colon \Cat\mbox{-}\mathbf{Multicat}\to \Cat\mbox{-}\mathbf{Multicat^{ps}}$ and the 2-functor which we define next.
\begin{defi}
We define the 2-functor $\psi\colon \Cat\mbox{-}\mathbf{Multicat^{ps}}\to \Cat\mbox{-}\mathbf{Multicat}$ as follows. For a $\Cat$-multicategory $\M,$ $\psi\M=\M\times E\Sigma_*.$ For $\M,\Nn$ $\Cat$-multicategories, we define the component functor $\psi$ as the composite
\begin{center}
    \begin{tikzcd}
        \Cat\mbox{-}\mathbf{Multicat^{ps}}(\M,\Nn)\ar[r,"{\eta_{\mathcal{N}}}_*"]\ar[rd,"\psi"']&\Cat\mbox{-}\mathbf{Multicat^{ps}}(\M,\Nn\times E\Sigma_*)\ar[d,"\phi"]\\
        &\Cat\mbox{-}\mathbf{Multicat}(\M\times E\Sigma_*,\Nn\times\Sigma_*).
    \end{tikzcd}
\end{center}
Thus, by \Cref{mainprop} if $F\colon\M\to\N$ is a pseudo symmetric $\Cat$-multifunctor, then $\psi F\colon\M\times E\Sigma_*\to\N\times E\Sigma_*$ is the unique symmetric $\Cat$-multifunctor which makes the diagram 
\begin{equation}\label{psi1}
    \begin{tikzcd}
        \M\ar[r,"\eta_{M}"]\ar[d,"F"']&\M\times E\Sigma_*\ar[d,"j\psi F"]\\
        \N\ar[r,"\eta_{\Nn}"']&\Nn\times E\Sigma_*
    \end{tikzcd}
\end{equation}
commute in $\Cat\mbox{-}\mathbf{Multicat^{ps}}.$ Similarly, by \Cref{2celllemma}, for $\theta\colon F\to G$ a pseudo symmetric $\Cat$-multinatural  transformation between $F,G\colon \M\to\Nn$ pseudo symmetric $\Cat$-multifunctors,  $\psi \theta\colon \psi F\to \psi G$ is the unique symmetric $\Cat$-multinatural transformation such that the equality of pasting diagrams

\begin{equation}\label{psi2}
\begin{tikzcd}
\M\arrow[dd, bend left=57, "G"{name=U}]\arrow[dd, bend right=57, "F"'{name=D}]\ar[r,"\eta_\M"]&\M\times E\Sigma_*\ar[dd,"j\psi F"]&& \M \ar[r,"\eta_\M"]\ar[dd,"F"']&\M\times E\Sigma_* \arrow[dd, bend left=50, "j\psi G"{name=A}]\arrow[dd, bend right=50, "j\psi F"'{name=B}]\\[-12pt]
&&\mathbf{=}&&\\[-12pt]
\Nn\ar[r,"\eta_\Nn"']&\N\times E\Sigma_*&&\Nn\ar[r,"\eta_\Nn"']&\N\times E\Sigma_*
\arrow[Rightarrow, from=D, to=U,shorten <=2mm,shorten >=1mm,"\theta"]
\arrow[Rightarrow, from=B, to=A,shorten <=2mm,shorten >=1mm,"j\psi \theta"]
\end{tikzcd} 
\end{equation}
holds in $\Cat\mbox{-}\mathbf{Multicat^{ps}}.$
\end{defi}
\begin{theorem}\label{finaladjunction}
    There is a 2-adjunction 
    \begin{center}
        \begin{tikzcd}
            \Cat\mbox{-}\mathbf{Multicat^{ps}}\ar[rr,bend left=10,"\psi"]&[-10pt]\bot&[-10pt]\Cat\mbox{-}\mathbf{Multicat}\ar[ll,bend left=10,"j"]
        \end{tikzcd}
    \end{center}
    where $j$ is the inclusion 2-functor.
\end{theorem}
\begin{proof}
Following \Cref{phi}, we define the unit of the adjunction as the strict 2-natural transformation $\eta\colon 1_{\Cat\mbox{-}\mathbf{Multicat^{ps}}}\to j\psi$ having component $\eta_\M$ at a $\Cat$-multicategory $\M.$ We also define the counit of the adjunction $\pi\colon \psi j\to 1_{\Cat\mbox{-}\mathbf{Multicat}}$ as having component at a $\Cat$-multicategory $\M$ the projection $\pi_M\colon \M\times E\Sigma_*\to \M.$\\

The fact that $\eta$ defines a strict 2-natural transformation follows directly from \Cref{psi1} and \Cref{psi2}. To prove that the data of $\pi$ defines a strict 2-natural transformation we need to prove that given $F\colon \M\to \N$ symmetric $\Cat$-multifunctor, the following diagram commutes:
\begin{center}
    \begin{tikzcd}
        \M\times E\Sigma_*\ar[r,"\pi_\M"]\ar[d,"\psi j F"']&\M\ar[d,"F"]\\
        \Nn\times E\Sigma_*\ar[r,"\pi_\Nn"']&\Nn.
    \end{tikzcd}
\end{center}
Indeed, we prove that $\psi jF=F\times 1_{E\Sigma_*}.$ By \Cref{psi1}, it suffices to show that the following diagram commutes in $\Cat\mbox{-}\mathbf{Multicat^{ps}}$:
\begin{equation}\label{symmf}
    \begin{tikzcd}
        \M\ar[r,"\eta_\M"]\ar[d,"jF"']&\M\times E\Sigma_*\ar[d,"j(F\times 1)"]\\
        \Nn\ar[r,"\eta_\N"']&\Nn\times E\Sigma_*.
    \end{tikzcd}
\end{equation}
It is clear that this diagram commutes at the level of objects, 1-cells, and 2-cells of the multicategory. The pseudo symmetry isomorphisms of both composites also agree. Indeed, for $f\colon \la a\ra\to b$ a 1-cell of $\M$ and $\sigma\in\Sigma_n,$ by \Cref{horcomp}, we get that 
\begin{align*}
    (j(F\times 1)\eta_\M)_{\sigma;f}=&j(F\times 1)_{\sigma;\eta_\M(f)}\circ j(F\times 1)({\eta_\M}_{\sigma;f})\\
    &=(1_{(Ff)\sigma},1_\sigma)\circ (1_{(Ff)\sigma},E_\id^\sigma)\\
    &=(1_{(Ff)\sigma},E_\id^\sigma)\circ (1_{(Ff)\sigma},1_\sigma)\\
    &={\eta_\N}_{\sigma;Ff}\circ\eta_\N(jF_{\sigma;f})\\
    &=(\eta_\N\circ jF)_{\sigma;f}.
\end{align*}
To finish proving the 2-naturality of $\pi_\M$, we need to prove that given $\M,\N$ $\Cat$-multicategories, $F,G\colon \M\to\Nn$ $\Cat$-multifunctors and a $\Cat$-multinatural transformation $\theta\colon F\to G,$  the following equality of pasting diagrams holds in $\Cat\mbox{-}\mathbf{Multicat}$:
\begin{equation*}
\begin{tikzcd}
\M\times E\Sigma_*\arrow[dd, bend left=50, "j(G\times 1)"{name=U}]\arrow[dd, bend right=50, "j(F\times 1)"'{name=D}]\ar[r,"\pi\M"]&\M\ar[dd," G"]&& \M\times E\Sigma_* \ar[r,"\pi_\M"]\ar[dd,"j(F\times 1)"']&\M\arrow[dd, bend left=60, "G"{name=A}]\arrow[dd, bend right=60, "F"'{name=B}]\\[-12pt]
&&\mathbf{=}&&\\[-12pt]
\Nn\times E\Sigma_*\ar[r,"\pi_\Nn"']&\N&&\Nn\times E\Sigma_*\ar[r,"\pi_\Nn"']&\N.
\arrow[Rightarrow, from=D, to=U,shorten <=2mm,shorten >=1mm,"\psi j \theta"]
\arrow[Rightarrow, from=B, to=A,shorten <=3mm,shorten >=2mm,"\theta"]
\end{tikzcd} 
\end{equation*}
In turn, the last equality of pasting diagrams holds since $\psi j \theta=j(\theta\times 1).$ To see this, by \Cref{psi2}, we must show the following equality of pasting diagrams in $\Cat\mbox{-}\mathbf{Multicat^{ps}}:$
\begin{equation}\label{eta2}
\begin{tikzcd}
\M\arrow[dd, bend left=55, "jG"{name=U}]\arrow[dd, bend right=55, "jF"'{name=D}]\ar[r,"\eta_\M"]&\M\times E\Sigma_*\ar[dd,"j( G\times 1)"]&& \M \ar[r,"\eta_\M"]\ar[dd,"jF"']&\M\times E\Sigma_* \arrow[dd, bend left=50, "j(G\times 1)"{name=A}]\arrow[dd, bend right=50, "j( F\times 1)"'{name=B}]\\[-12pt]
&&\mathbf{=}&&\\[-12pt]
\Nn\ar[r,"\eta_\Nn"']&\N\times E\Sigma_*&&\Nn\ar[r,"\eta_\Nn"]&\N\times E\Sigma_*.
\arrow[Rightarrow, from=D, to=U,shorten <=2mm,shorten >=1mm,"j\theta"]
\arrow[Rightarrow, from=B, to=A,shorten <=2mm,shorten >=1mm,"j( \theta\times 1)"]
\end{tikzcd} 
\end{equation}
To check that this equality holds let $a\in\ob(\M).$ We get, by \Cref{hor}, that
\begin{align*}
    (1_{\eta_\Nn}*j\theta)_a&=\gamma\left(1_{\eta_\N(jGa)};\eta_\Nn (\theta_a)\right)\\
    &=\gamma\left((1_{Ga},1_{\id});(\theta_a,1_\id)\right)\\
    &=\gamma\left((\theta_a,1_\id);(1_{Fa},1_\id)\right)\\
    &=\gamma\left(j(\theta\times 1)_{\eta_\Nn (a)};j(F\times 1)(1_{\eta_\M(a)})\right)\\
    &=(j(\theta\times 1)*\eta_\M)_a.
\end{align*}
Thus, $\eta$ and $\pi$ are strict 2-natural transformations and we just need to prove that they satisfy the triangle identities. To prove that the identity $(1_j*\pi)(\eta*1_j)=1_j$ holds we need to prove that for $\M$ a $\Cat$-multicategory the diagram
\begin{center}
    \begin{tikzcd}
        &\M\times E\Sigma_*\ar[rd,"j\pi_\M"]&\\
        \M\ar[ru,"\eta_\M"]\ar[rr,"1_\M"']&&\M
    \end{tikzcd}
\end{center}
commutes in $\Cat\mbox{-}\mathbf{Multicat^{ps}}$. This is clear at the level of objects, $n$-ary 1-cells and $n$-ary 2-cells. The pseudo symmetry isomorphisms of both pseudo symmetric $\Cat$-multifunctors also agree since, for  $f\colon \la a\ra\to b$  an $n$-ary 1-cell of $\M$ and $\sigma\in\Sigma_n,$ we obtain, by \Cref{horcomp},
$$((j\pi_\M)\circ\eta_\M)_{\sigma;f}=(j\pi_\M)_{\sigma;\eta_{\M}(f)}\circ j\pi_\M ({\eta_{M}}_{\sigma;f})=1_{f\sigma}={1_{\M}}_{\sigma;f}.$$
The other triangle identity is $(\pi*1_\psi)(1_\psi*\eta)=1_\psi.$ To check it, we must prove that, given a $\Cat$-multicategory $\M,$  the composite
\begin{center}
    \begin{tikzcd}
        \M\times E\Sigma_*\ar[r,"\psi\eta_\M"]&\M\times E\Sigma_*\times E\Sigma_*\ar[r,"\pi_{\M\times E\Sigma_*}"]&\M\times E\Sigma_*
    \end{tikzcd}
\end{center}
agrees with $1_{\M\times E\Sigma_*}.$ This holds since, if $\Delta\colon E\Sigma_*\to E\Sigma_*\times E\Sigma_*$ denotes the diagonal map, then $\psi(\eta_\M)=1_\M\times \Delta.$ To see this, notice that by \Cref{psi1} all we need is to prove that the following diagram is commutative:
\begin{equation}\label{eta}
    \begin{tikzcd}
        \M\ar[r,"\eta_\M"]\ar[d,"\eta_\M"']&\M\times  E\Sigma_*\ar[d,"j(1\times \Delta)"]\\
        \M\times E\Sigma_*\ar[r,"\eta_{\M\times  E\Sigma_*}"']&\M\times E\Sigma_* \times E\Sigma_*.
    \end{tikzcd}
\end{equation}
Now, the previous diagram is evidently commutative at the level of objects, 1-cells, and 2-cells. The diagram also commutes at the level of pseudo symmetry isomorphisms  since, for $f\colon\la a\ra\to b$ an $n$-ary 1-cell in $\M$ and $\sigma\in\Sigma_n,$
\begin{align*}
    {\left(\eta_{\M\times E\Sigma_*}\circ\eta_\M\right)}_{\sigma;f}&={\eta_{\M\times E\Sigma_*}}_{\sigma;\eta_\M(f)}\circ \eta_{\M\times E\Sigma_*}({\eta_M}_{\sigma;f})\\
    &=(1_{f\sigma},1_\sigma,E_\id^\sigma)\circ(1_{f\sigma},E_\id^\sigma,1_\id)\\
    &=(1_{f\sigma},1_\sigma,1_\sigma)\circ(1_{f\sigma},E_\id^\sigma,E_\id^\sigma)\\
    &=j(1\times\Delta)_{\sigma;\eta_\M(f)}\circ j(1\times \Delta)({\eta_{\M}}_{\sigma;f})\\
    &=\left(j(1\times \Delta)\circ \eta_\M\right)_{\sigma;f}.
\end{align*}
We conclude that the triangle identities are satisfied and thus we get the desired 2-adjunction.
\end{proof}
We can use this 2-adjunction to describe the 2-category $\mathbf{Cat}$-$\mathbf{Multicat^{ps}}$ in terms of symmetric \textbf{Cat}-multifunctors and symmetric \textbf{Cat}-multinatural transformations alone, thus upgrading the functors $\phi$ from \Cref{phi}  to an isomorphism of 2-categories.

\begin{defi}\label{2D}The 2-category $\mathbf{D}$ has $\textbf{Cat}$-multicategories as objects. For $\mathcal{M},\mathcal{N}$ $\mathbf{Cat}$-multicategories, the category of morphisms between $\M$ and $\N$ is 
$$\mathbf{D}(\mathcal{M},\mathcal{N})=\mathbf{Cat}\mbox{-}\mathbf{Multicat}(\mathcal{M}\times E\Sigma_*,\mathcal{N}).$$
In particular, vertical composition of 2-cells is defined as in $\Cat\mbox{-}\mathbf{Multicat}.$ For $F\colon\mathcal{M}\times E\Sigma_*\to \mathcal{N}$ and $G\colon\mathcal{N}\times E\Sigma_*\to \mathcal{Q}$ symmetric $\Cat$-multifunctors, the composition $G\circ{F}$ is defined as the composite
\begin{center}
    \begin{tikzcd}
        \M\times E\Sigma_*\ar[r,"1\times \Delta"]&\M\times E\Sigma_*\times E\Sigma_*\ar[r,"F\times 1"]&\Nn\times E\Sigma_*\ar[r,"G"]&\mathcal{Q}
    \end{tikzcd}
\end{center}
in $\Cat\mbox{-}\mathbf{Multicat}.$ Similarly, for $F,J\colon\M\times E\Sigma_* \to \N,$ $G,K\colon\N\times E\Sigma_*\to \mathcal{Q}$ symmetric $\Cat$-multifunctors and $\theta\colon F\to J,$ $\zeta\colon G\to K$ $\Cat$-multinatural transformations, $\zeta *\theta$ is defined as the pasting
\begin{center}
    \begin{tikzcd}
        \M\times E\Sigma_*\ar[r,"1\times \Delta"]&\M\times E\Sigma_*\times E\Sigma_*\ar[r,bend left=15,"F\times 1"{name=F}]\ar[r,bend right=15,"J\times 1"{swap,name=J}]&[15pt]\Nn\times E\Sigma_*\ar[r,bend left=30,"G"{name=G}]\ar[r,bend right=30,"K"{swap,name=K}]&\mathcal{Q}
        \arrow[Rightarrow,from=F,to=J,shorten <=2mm,shorten >=2mm,"\theta\times 1"]
        \arrow[Rightarrow,from=G,to=K,shorten <=2mm,shorten >=2mm,"\zeta"']
    \end{tikzcd}
\end{center}
in $\Cat\mbox{-}\mathbf{Multicat}.$
\end{defi}
The previous definition makes $\mathbf{D}$ into a 2-category and the functors $\phi,$ and $\eta^*$ from \Cref{phi} into the components of isomorphisms of 2-categories.
\begin{theorem}The data of the previous definition defines a 2-category $\mathbf{D}$ isomorphic to  $\Cat\mbox{-}\mathbf{Multicat^{ps}}.$
\end{theorem}
\begin{proof}
The (horizontal) composition functors are defined so that $\phi$ and $\eta^*$ become the componentwise functors of a 2-category isomorphism between $\mathbf{D}$ and $\mathbf{Cat}$-$\mathbf{Multicat^{ps}}.$ More precisely, for $\mathcal{M},\mathcal{N}$ and $\mathcal{Q}$  \textbf{Cat}-multicategories, we will prove that the $\mathbf{D}$ composition functor defined, $\circ ' \colon\mathbf{D}(\mathcal{N},\mathcal{Q})\times \mathbf{D}(\mathcal{M},\mathcal{N})\to\mathbf{D}(\mathcal{M},\mathcal{Q}),$ makes the following diagram commute, where $\circ$ denotes the horizontal composition functor of $\mathbf{Cat\mbox{\textbf{-}}Multicat^{ps}}\colon$
\begin{equation}\label{isomorphism}
\begin{tikzcd}[font=\small]
    \mathbf{D}(\mathcal{N},\mathcal{Q})\times\mathbf{D}(\mathcal{M},\mathcal{N})\ar[r,"\circ'"]\ar[d,"\eta*\times \eta*"swap]&[-10pt]\mathbf{D}(\mathcal{M},\mathcal{Q})\\
   \mathbf{Cat\mbox{-}Multicat^{ps}}(\mathcal{N},\mathcal{Q})\times\mathbf{Cat\mbox{-}Multicat^{ps}}(\mathcal{M},\mathcal{N})\ar[r,"\circ"swap]&\mathbf{Cat\mbox{-}Multicat^{ps}}(\mathcal{M},\mathcal{Q}).\ar[u,"\phi"swap]
\end{tikzcd}
    \xymatrix{}
\end{equation}
Let $G\colon\Nn\times \mathcal{Q}$ and $F\colon\M\times E\Sigma_*\to \Nn$ be symmetric $\Cat$-multifunctors. The commutativity of \Cref{isomorphism} for $(G,F)$ reduces to the commutativity of the following diagram by \Cref{mainprop}: 
\begin{center}
    \begin{tikzcd}
        \M\ar[r,"\eta_\M"]\ar[d,"\eta_\M"']&\M\times E\Sigma_*\ar[d,"j(1\times \Delta)"]&\\
        \M\times E\Sigma_*\ar[r,"\eta_{\M\times E\Sigma_*}"]\ar[d,"jF"']&\M\times E\Sigma_*\times E\Sigma_*\ar[d,"j(F\times 1)"]&\\
        \N\ar[r,"\eta_N"]&\N\times E\Sigma_*\ar[r,"jG"]&\mathcal{Q}.
    \end{tikzcd}
\end{center}
This diagram in turn is commutative by \Cref{symmf} and \Cref{eta}. Now, if $F,G$ are as before, $J\colon\M\times E\Sigma_*$ and $K\colon\N\times E\Sigma_*\to \mathcal{Q}$ are symmetric $\Cat$-multifunctors, and $\theta\colon F\to J,$ $\zeta\colon G\to K$ are $\Cat$-multinatural transformations, by \Cref{2celllemma}, the commutativity of \Cref{isomorphism} for $(\zeta,\theta)$ reduces to the equality of pasting diagrams:
\begin{center}
    \begin{tikzcd}
        \M\ar[r,"\eta_\M"]\ar[d,"\eta_\M"']&\M\times E\Sigma_*\ar[d,"j(1\times \Delta)"]&[-10pt]&[-10pt]\M\ar[d,"\eta_\M"']\ar[r,"\eta_\M"]&\M\times E\Sigma_*\ar[d,"j(1\times \Delta)"]\\
        \M\times E\Sigma_*\ar[r,"\eta_{\M\times E\Sigma_*}"]\ar[d,"jF"']&\M\times E\Sigma_*\times E\Sigma_*\ar[d,bend right=40,"j(F\times 1)"{swap,name=F}]\ar[d,bend left=40,"j(J\times 1)"name=J]&=&\M\times E\Sigma_*\ar[r,"\eta_{\M\times E\Sigma_*}"]
        \ar[d,bend right=40,"jF"{swap,name=f}]\ar[d,bend left=40,"jJ"{name=j}]&\M\times E\Sigma_*\times E\Sigma_*\ar[d,"j(J\times 1)"]\\[19pt]
        \N\ar[r,"\eta_N"]&\N\times E\Sigma_*\ar[d,bend right=40,"jG"{swap,name=G}]\ar[d,bend left=40,"jK"name=K]&&\N\ar[r,"\eta_N"]&\N\times E\Sigma_*\ar[d,bend right=40,"jG"{swap,name=g}]\ar[d,bend left=40,"jK"name=k]\\[19pt]
        &\mathcal{Q}&&&\mathcal{Q}.
        \arrow[Rightarrow,from=F,to=J,shorten <=2mm,shorten >=2mm,"j(\theta\times 1)"]
        \arrow[Rightarrow,from=f,to=j,shorten <=2mm,shorten >=2mm,"j\theta"]
        \arrow[Rightarrow,from=G,to=K,shorten <=2mm,shorten >=2mm,"j\zeta "]
        \arrow[Rightarrow,from=g,to=k,shorten <=2mm,shorten >=2mm,"j\zeta "]
    \end{tikzcd}
\end{center}
This equality holds by \Cref{eta2} and makes implicit use of \Cref{symmf} and \Cref{eta}. We can thus define $\phi\colon \mathbf{Cat}\mbox{\textbf{-}}\mathbf{Multicat^{ps}}\to \mathbf{D}$ in objects as the identity map, and do the same for $\eta^*\colon \mathbf{D}\to \mathbf{Cat}\mbox{\textbf{-}}\mathbf{Multicat^{ps}},$ with the component functors given for $\M$ and $\N$ multicategories by \Cref{phimn} and \Cref{etamn} respectively. By \Cref{isomorphism} and the fact that $\phi$ and $\eta^*$ are componentwise isomorphisms, $\phi$ and $\eta$ preserve vertical composition of 2-cells and horizontal composition of 1-cells and 2-cells. The fact that $\mathbf{Cat}\mbox{\textbf{-}}\mathbf{Multicat^{ps}}$ is a 2-category implies that $\mathbf{D}$ is a 2-category. This further turns $\phi$ and $\eta^*$ into isomorphisms of 2-categories.
\end{proof}
\section{Applications to inverse \texorpdfstring{$K$}-theory}\label{Section4}
We use our understanding of pseudo symmetric multifunctors to show that they preserve $E_n$-algebras for $n=1,2,3,...,\infty.$ First we define $E_n$ $\Cat$-operads.
\begin{defi}
For $n=1,...,\infty,$ an $E_n$ $\Cat$-operad is a $\Cat$-operad that becomes a topological $E_n$-operad (in the sense of \cite{May72}) after applying the classifying space functor. A topological $E_n$ operad is one that has the same $\Sigma$-equivariant homotopy type as the little $n$-cubes operad.
\end{defi}
\begin{ex}
    An example of an $E_\infty$ $\Cat$-operad is $E\Sigma_*.$  There are also examples of $E_n$ $\Cat$-operads for each $n=1,2,\dots$ in \cite{Berger96} and \cite{Fiedorowicz03}, which furthermore have a free action of the symmetric group (on objects).
\end{ex}
\begin{defi}
Let $\M$ be a $\Cat$-multicategory and $\mathcal{O}$ a $\Cat$-operad. An algebra (respectively a pseudo symmetric algebra) in $\M$ over $\mathcal{O}$ is a symmetric (respectively pseudo symmetric) $\Cat$-multifunctor $\mathcal{O}\to \M.$ For $n\in\{1,2,\dots,\infty\},$ an $E_n$ algebra (respectively a pseudo symmetric $E_n$ algebra) in $\mathcal{M}$ is an algebra (respectively a pseudo symmetric algebra) over an $E_n$ operad. 
\end{defi}
\begin{rmk}If $\mathcal{O}$ is $\mathbf{Cat}$-operad and $\M$ is a $\Cat$-multicategory,  the pseudo symmetric algebras over $\mathcal{O}$ agree with symmetric algebras over the operad $\mathcal{O}\times E\Sigma_*.$ For example, while algebras over the commutative operad $\{*\}$ in $\M$ are the commutative monoids in $\M$, pseudo symmetric algebras over $\{*\}$ in $\M$ are precisely algebras over the Barratt-Eccles operad and thus, $E_\infty$-algebras. Similarly, pseudo symmetric algebras over the $E_\infty$ $\Cat$-operad $E\Sigma_*,$ which are defined in \cite{Y23} as pseudo symmetric $E_\infty$ algebras in $\M$, are algebras over $E\Sigma_*\times E\Sigma_*=E(\Sigma_*\times \Sigma_*)$ which is still an $E_\infty$ $\Cat$-operad, and thus, they are still $E_\infty$ algebras in the sense defined above. If we let $\mathcal{O}$ be a symmetric $\Cat$-operad with a free action of the symmetric group, $\mathcal{O}\times E\Sigma_*$ is componentwise $\Sigma$-equivariantly homotopy equivalent to $\mathcal{O}$ (after taking nerves), that is, for each $n\geq 0,$ the projection $\mathcal{O}(n)\times E\Sigma_n\to \mathcal{O}(n)$ induces a $\Sigma_n$ equivariant homotopy equivalence on nerves. Thus, we have the following result.
\end{rmk}
\begin{lem}{\color{white}.}
\begin{enumerate}
    \item Let $\mathcal{O}$ be a ($\Sigma$-free) $E_n$ $\Cat$-operad. Then $\mathcal{O}\times E\Sigma_*$ is an $E_n$ $\Cat$-operad.
    \item Pseudo symmetric $E_n$ algebras over ($\Sigma$-free) $E_n$ $\Cat$-operads are $E_n$ algebras for $n=1,2,\dots,\infty.$
\end{enumerate}
\end{lem}
We remind the reader that freeness is not a serious restriction since there are $E_n$ operads in $\Cat,$ like those in \cite{Berger96} and \cite{Fiedorowicz03} which are free. As a corollary we conclude that pseudo symmetric $\Cat$-multifunctors preserve $E_n$ algebras.
\begin{cor}\label{finalcor}
    Let $\M$ and $\N$ be $\Cat$-multicategories and $F\colon \mathcal{\M}\to \mathcal{N}$ be a pseudo symmetric $\Cat$-multifunctor, then:
    \begin{enumerate}
        \item $F$ sends commutative monoids in $\M$ to $E_\infty$ algebras in $\N.$
        \item $F$ sends $E_n$-algebras (parameterized by free $\Cat$-operads), to $E_n$-algebras.
    \end{enumerate}
\end{cor}
We conclude our paper by applying our understanding of pseudo symmetric $\Cat$-multifunctors to multifunctorial inverse $K$-theory. In \cite{JY22}, Johnson and Yau define Mandell's inverse $K$-theory multifunctor $\mathcal{P}$ as well as the $\Cat$-multicategories that are its domain ($\Gamma$-categories) and target (permutative categories). Yau proves in \cite{Y23} that $\mathcal{P}$ is pseudo symmetric. We refer the interested reader \cite{Y23} of which the following theorem is one of the main results.
\begin{theorem}\normalfont{\cite{Y23}}
    Mandell's inverse $K$-theory functor is a pseudo symmetric $\Cat$-multifunctor $\mathcal{P}\colon\Gamma\mbox{-}\Cat\to \mathbf{PermCat^{sg}}.$
\end{theorem}
As a consequence, $\mathcal{P}$ sends commutative monoids to $E_\infty$ algebras and preserves $E_n$ algebras, as was stated in \Cref{maincor}.
\bibliographystyle{alpha}
\bibliography{repccorresp.bib}
\end{document}